\documentclass[10pt,a4paper,oneside]{amsart}

\usepackage{amsthm,amsfonts,amsmath,amssymb,verbatim} 
\usepackage[width=13cm]{geometry}
\usepackage{mathtools}
\usepackage{tikz,graphicx}
\usepackage{tikz-cd}
\usepackage{mathrsfs}
\usepackage{fouridx}

\usepackage[normalem]{ulem}

\usepackage{hyperref}

\numberwithin{equation}{section}
\newtheorem{theorem}{Theorem}[section]
\newtheorem{coro}[theorem]{Corollary}
\newtheorem{prop}[theorem]{Proposition}
\newtheorem{lemma}[theorem]{Lemma}

\renewcommand{\leq}{\leqslant}
\renewcommand{\geq}{\geqslant}

\newcommand{\xir}[1]{\xi^{(#1)}}
\newcommand{\Mod}[1]{\ (\mathrm{mod}\ #1)}
\newcommand{\rk}{\mathrm{rk}}
\renewcommand{\o}{\mathrm{o}}
\newcommand{\rs}{\mathrm{rs}}
\newcommand{\so}{\mathrm{so}}
\renewcommand{\sp}{\mathrm{sp}}

\newcommand{\abs}[1]{\lvert#1\rvert}

\newcommand{\la}{\lambda}
\newcommand{\La}{\Lambda}
\DeclareMathOperator{\sgn}{sgn}
\newcommand{\mur}[1]{\mu^{(#1)}}
\newcommand{\nur}[1]{\nu^{(#1)}}
\newcommand{\lar}[1]{\lambda^{(#1)}}
\newcommand{\tcore}{t\textup{-core}}
\newcommand{\xcore}[1]{#1\text{-core}}

\begin{document}

\title[Universal characters]{Universal characters twisted by roots of unity}
\date{}
\author{Seamus P.~Albion}
\address{Fakult\"{a}t f\"{u}r Mathematik, Universit\"{a}t Wien, 
Oskar-Morgenstern-Platz 1, A-1090 Vienna, Austria}
\email{seamus.albion@univie.ac.at}


\begin{abstract}
A classical result of Littlewood gives a factorisation for the Schur function
at a set of variables ``twisted'' by a primitive $t$-th root of unity, 
characterised by the core and quotient of the indexing partition.
While somewhat neglected, it has proved to be an important tool 
in the character 
theory of the symmetric group, the cyclic sieving phenomenon, plethysms of
symmetric functions and more.
Recently, similar factorisations for the characters of the groups
$\mathrm{O}(2n,\mathbb{C})$, $\mathrm{Sp}(2n,\mathbb{C})$ and 
$\mathrm{SO}(2n+1,\mathbb{C})$ were obtained by Ayyer and Kumari.
We lift these results to the level of universal characters, which 
has the benefit of making the proofs simpler and the structure of the 
factorisations more transparent.
Our approach also allows for universal character extensions of some
factorisations of a different nature originally discovered
by Ciucu and Krattenthaler, and generalised by Ayyer and Behrend.
\smallskip

\end{abstract}

\maketitle

\section{Introduction}

In his $1940$ book \emph{The Theory of Group Characters and Matrix 
Representations of Groups}, D.~E.~Littlewood devotes a section to the 
evaluation of the Schur function $s_\la$ at a set of variables ``twisted''
(not his term) by a primitive $t$-th root of unity $\zeta$ 
\cite[\S7.3]{Littlewood40}.
In modern terminology,
Littlewood's theorem asserts that $s_\la$ evaluated at the
variables $\zeta^jx_i$ for $1\leq i\leq n$ and $0\leq j\leq t-1$ is
zero unless the $t$-core of $\la$ is empty.
Moreover, when it is nonzero, it factors as a product of Schur functions
indexed by the elements of the $t$-quotient of $\la$, each with the 
variables $x_1^t,\dots,x_n^t$.

The Schur functions are characters of the irreducible
polynomial representations of
the general linear group $\mathrm{GL}(n,\mathbb{C})$.
Ayyer and Kumari \cite{AK22} have recently generalised
Littlewood's theorem to the characters of the other classical groups 
$\mathrm{O}(2n,\mathbb{C})$, $\mathrm{Sp}(2n,\mathbb{C})$ and 
$\mathrm{SO}(2n+1,\mathbb{C})$ indexed by partitions.
While their factorisations are still indexed by the $t$-quotient of the 
corresponding partition, the vanishing is governed by the $t$-core having a 
particular form.
More precisely, $\tcore(\la)$ is of the form $(a\mid a+z)$ in Frobenius 
notation, where $z=-1,1,0$, for 
$\mathrm{O}(2n,\mathbb{C})$, $\mathrm{Sp}(2n,\mathbb{C})$ and 
$\mathrm{SO}(2n+1,\mathbb{C})$ respectively.
Note that these are the same partitions occurring in Littlewood's Schur 
expansion of the Weyl denominators for types $\mathrm{B}_n$, $\mathrm{C}_n$
and $\mathrm{D}_n$ \cite[p.~238]{Littlewood40} 
(see also \cite[p.~79]{Macdonald95}).

Littlewood's proof, and the proofs of Ayyer and Kumari, use the Weyl-type
expressions for the characters as ratios of alternants.
In the Schur case, Chen, Garsia and Remmel \cite{CGR84} and independently
Lascoux \cite[Theorem~5.8.2]{Lascoux?} have given an alternate 
proof based on the Jacobi--Trudi formula \eqref{Eq_JT}.
This approach was already known to Farahat, who used it to extend Littlewood's
theorem to skew Schur functions $s_{\la/\mu}$ where $\mu$ is the $t$-core of 
$\la$ \cite[Theorem~2]{Farahat58}.
The full skew Schur case was then given by Macdonald \cite[p.~91]{Macdonald95}, 
again proved using the Jacobi--Trudi formula; see Theorem~\ref{Thm_skew}
below.

In this article we lift the results of Ayyer and Kumari to the 
much more general universal characters of the groups
$\mathrm{O}(2n,\mathbb{C})$, $\mathrm{Sp}(2n,\mathbb{C})$ and 
$\mathrm{SO}(2n+1,\mathbb{C})$ as defined by Koike and Terada \cite{KT87}.
These are symmetric functions indexed 
by partitions which, under appropriate specialisation of the variables,
become actual characters of their respective groups.
In fact, these generalise the Jacobi--Trudi-type
formulas for the characters of these groups, which were first written down
by Weyl \cite[Theorems~7.8.E \& 7.9.A]{Weyl39}.
For the universal characters we generalise the notion of ``twisting''
a set of variables by introducing operators 
$\varphi_t:\La\longrightarrow\La$ for each integer $t\geq 2$ 
which act on the complete homogeneous symmetric functions as 
\begin{equation}\label{Eq_phi-def}
\varphi_t h_r=\begin{cases}h_{r/t} &\text{if $t$ divides $r$}, \\
0 &\text{otherwise}.\end{cases}
\end{equation}
It is not at all hard to show that the image of $\varphi_t$ acting on a
symmetric function at the variables $x_1^t,\dots,x_n^t$ agrees with the result 
of twisting the variables $x_1,\dots,x_n$ by $\zeta$.
The advantages of this framework for such factorisations are that the 
proofs are much simpler, and the structure of the factorisations is made
transparent.
Moreover, we are able to discuss dualities between these objects 
which are only present at the universal level.
A particularly important tool for our purposes is Koike's universal character 
$\rs_{\la,\mu}$ \eqref{Eq_rs-def} associated with 
a rational representation of $\mathrm{GL}(n,\mathbb{C})$.
This object, which is used later in Subsection~\ref{Sec_CK} to prove 
other character factorisations, 
appears to be the correct universal character analogue
of the Schur function with variables $(x_1,1/x_1,\dots,x_n,1/x_n)$.

The remainder of the paper reads as follows.
In the next section we outline the preliminaries on partitions and symmetric
functions needed to state our main results, which follow in 
Section~\ref{Sec_results}.
In the following Section~\ref{Sec_lemmas} we prepare for the proofs of these 
results by giving a series of lemmas regarding cores and quotients and their
associated signs.
The factorisations are then proved in Section~\ref{Sec_proofs},
including a detailed proof of the Schur case, following Macdonald.
The final Section~\ref{Sec_stuff} concerns other factorisation results 
relating to Schur functions and other characters.
This includes universal extensions of factorisations very different
from those already discussed originally due to Ciucu and Krattenthaler, later 
generalised by Ayyer and Behrend.

\section{Preliminaries}
\subsection{Partitions}
A \emph{partition} $\la=(\la_1,\la_2,\la_3,\dots)$ is a weakly 
decreasing sequence of nonnegative integers such that only finitely many of the
$\la_i$ are nonzero. 
The nonzero $\la_i$ are called \emph{parts} and the number of parts the 
\emph{length}, written $l(\la)$.
We say $\la$ is \emph{a partition of $n$} if 
$\abs{\la}:=\la_1+\la_2+\la_3+\cdots=n$.
Two partitions are regarded as the same if they agree up to trailing zeroes,
and the set of all partitions is written $\mathscr{P}$.
A partition is identified with its \emph{Young diagram},
which is the left-justified array of squares consisting of $\la_i$ squares in 
row $i$ with $i$ increasing downward.
For example
\smallskip
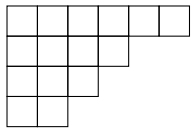
\begin{center}
\begin{tikzpicture}[scale=0.4]
\foreach \i [count=\ii] in {6,4,3,2}
\foreach \j in {1,...,\i}{\draw (\j,1-\ii) rectangle (\j+1,-\ii);}
\end{tikzpicture}
\end{center}
is the Young diagram of $(6,4,3,2)$. 
We define the \emph{conjugate} partition $\la'$ by reflecting the diagram of 
$\la$ in the main diagonal $x=y$, so that the conjugate of $(6,4,3,2)$ above is
$(4,4,3,2,1,1)$.
If $\la=\la'$ then $\la$ is called \emph{self-conjugate}.
For a square at coordinate $(i,j)$ where $1\leq i\leq l(\la)$ and 
$1\leq j\leq \la_i$ the \emph{hook length} is $h(i,j)=\la_i+\la_j'-i-j+1$.
For example the square $(1,2)$ in
\smallskip
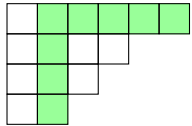
\begin{center}
\begin{tikzpicture}[scale=0.4]
\filldraw[color=green!40] (2,0) rectangle (7,-1);
\filldraw[color=green!40] (2,-1) rectangle (3,-4);
\foreach \i [count=\ii] in {6,4,3,2}
\foreach \j in {1,...,\i}{\draw (\j,1-\ii) rectangle (\j+1,-\ii);}
\end{tikzpicture}
\end{center}
has hook length $8$, with its hook shaded.
A partition $\la$ \emph{is a $t$-core} if it contains no squares of
hook length $t$, the set of which is denoted $\mathscr{C}_t$.
For a pair of partitions $\la,\mu$ we write $\mu\subseteq\la$ if the diagram 
of $\mu$ can be drawn inside the diagram of $\la$, i.e., if 
$\mu_i\leq \la_i$ for all $i\geq 1$.
In this case we can form the \emph{skew shape} $\la/\mu$ by removing
the digram of $\mu$ from that of $\la$.
For example $(3,2,1,1)\subseteq(6,4,3,2)$ and 
the diagram of $(6,4,3,2)/(3,2,1,1)$ is given by the non-shaded squares of 
\smallskip
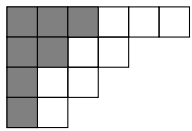
\begin{center}
\begin{tikzpicture}[scale=0.4]
\foreach \i [count=\ii] in {6,4,3,2}
\foreach \j in {1,...,\i}{\draw (\j,1-\ii) rectangle (\j+1,-\ii);}
\foreach \i [count=\ii] in {3,2,1,1}
\foreach \j in {1,...,\i}{\draw[fill=gray] (\j,1-\ii) rectangle (\j+1,-\ii);}
\end{tikzpicture}
\end{center}
A skew shape is called a \emph{ribbon} (or \emph{border strip}, \emph{rim hook},
\emph{skew hook}) if its diagram is connected 
and contains no $2\times 2$ square.
A $t$-ribbon is a ribbon with $t$ boxes. 
The \emph{height} of a $t$-ribbon $R$, written $\mathrm{ht}(R)$,
is one less than the number of rows it occupies.
In our example above $R=(6,4,3,2)/(3,2,1,1)$ is an $8$-ribbon with
height $\mathrm{ht}(R)=3$.
We say a skew shape is \emph{tileable by $t$-ribbons} or \emph{$t$-tileable} if 
there exists a sequence of partitions
\begin{equation}\label{Eq_RD}
\mu=\nur{0}\subseteq\nur{1}\subseteq\cdots\subseteq\nur{k-1}\subseteq\nur{k}
=\la
\end{equation}
such that $\nur{i}/\nur{i-1}$ is a $t$-ribbon for $1\leq i\leq k$.
A sequence $D=(\nur{0},\dots,\nur{k})$ (not to be confused with the
$t$-quotient of $\nu$ below, for which we use the same notation) 
satisfying \eqref{Eq_RD} is called a 
\emph{ribbon decomposition} (or \emph{border strip decomposition}) of $\la/\mu$.
We define the height of a ribbon decomposition to be the sum of the heights
of the individual ribbons:
$\mathrm{ht}(D):=\sum_{i=1}^k\mathrm{ht}(\nur{i}/\nur{i-1})$.
As shown by van Leeuwen \cite[Proposition~3.3.1]{vanLeeuwen99} and 
Pak \cite[Lemma~4.1]{Pak00} (also in \cite[\S6]{APRU21}), 
the quantity $(-1)^{\mathrm{ht}(D)}$ is the same for every ribbon decomposition
of $\la/\mu$.
We therefore define the \emph{sign} of a $t$-tileable skew shape $\la/\mu$ as
\begin{equation}\label{Eq_sgn-def}
\sgn_t(\la/\mu):=(-1)^{\mathrm{ht}(D)}.
\end{equation}

Let $\rk(\la)$ be the greatest integer such that $\rk(\la)\geq \la_{\rk(\la)}$,
usually called the \emph{Frobenius rank} of $\la$.
Equivalently, $\rk(\la)$ is the side length of the largest square which fits 
inside the diagram of $\la$ (the \emph{Durfee square}).
A partition can alternatively be written in \emph{Frobenius notation} as
\[
\la=\big(\la_1-1,\dots,\la_{\rk(\la)}-\rk(\la)\mid
\la'_1-1,\dots,\la'_{\rk(\la)}-\rk(\la)\big).
\]
Any pair of integer sequences 
$a_1>\cdots>a_k\geq0$ and $b_1>\cdots>b_k\geq0$ thus
determines a partition $\la=(a\mid b)$ with $\rk(\la)=k$.
For $z\in\mathbb{Z}$ and an integer sequence of predetermined length 
$a=(a_1,\dots,a_k)$ we write $a+z:=(a_1+z,\dots,a_k+z)$.
Following Ayyer and Kumari \cite[Definition~2.9]{AK22},
$\la$ is called \emph{$z$-asymmetric} if it is of the form $\la=(a\mid a+z)$ 
for some integer sequence $a$ and integer $z$.
Clearly a $0$-asymmetric partition is self-conjugate.
Partitions which are $-1$- and $1$-asymmetric are called
\emph{orthogonal} and \emph{symplectic} respectively.

\subsection{Cores and quotients}
We now describe the $t$-core and $t$-quotient of $\la$ arithmetically
following \cite[p.~12]{Macdonald95}.
There are many equivalent descriptions, see for instance 
\cite{GKS90,HJ11,JK81,WW20}.
We begin with the \emph{beta set} of a partition, 
which is simply the set of $n$ integers
\[
\beta(\la;n):=\{\la_1+n-1,\la_2+n-2,\dots,\la_{n-1}+1,\la_n\},
\]
where $n\geq l(\la)$ is fixed.
The number of elements in this set congruent to $r$ modulo 
$t$ is denoted by $m_r(\la;n)=m_r$.
Each element which falls into residue class $r$ for $0\leq r\leq t-1$ can
be written as $\xir{r}_kt+r$ for some integers 
$\xir{r}_1>\cdots>\xir{r}_{m_r}\geq 0$.
These integers are used to define a partition with parts
$\lar{r}_k=\xir{r}_k-m_r(\la;n)+k$ where $1\leq k\leq m_r(\la;n)$,
and the ordered sequence $(\lar{0},\dots,\lar{t-1})$ of these partitions
is called the \emph{$t$-quotient}.
The precise order of the constituents of the $t$-quotient depends on the 
residue class of $n$ modulo $t$. However, the orders only differ by cyclic
permutations, and Macdonald comments that it is best to think of the
quotient as a sort-of ``necklace'' of partitions.
To simplify things somewhat, we adopt the convention that the $t$-quotient is 
always computed with $n$ a multiple of $t$, so that the order of its 
constituents is fixed.
To define the $t$-core, one writes down the $n$ distinct integers $kt+r$ where
$0\leq k\leq m_r(\la;n)-1$ and $0\leq r\leq t-1$ in descending order,
say as $\tilde\xi_1>\dots>\tilde\xi_n$. Then $\tcore(\la)_i:=\tilde\xi_i-n+i$.
If $\tcore(\la)$ is empty then we say $\la$ \emph{has empty $t$-core}. 

It will prove useful later on to work with the
\emph{bead configurations} (or \emph{bead diagrams}, \emph{abacus model})
of James and Kerber \cite[\S2.7]{JK81}, which give a different model for
$t$-cores and $t$-quotients.
The ``board'' for a bead configuration is the set of nonnegative integers
arranged in $t$ downward-increasing columns, called \emph{runners}, 
according to their residues modulo $t$.
A bead is then placed at the space corresponding to each element of 
$\beta(\la;n)$.
For an example, let $\la=(4,4,3,2,1)$ so that
$\beta(\la;6)=\{9,8,6,4,2,0\}$.
Then the bead configuration for $\la$ with $t=3$ and $n=6$ and the beads
labelled by their position is
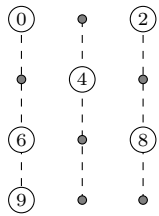
\begin{center}
\begin{tikzpicture}[scale=0.8]
\foreach \i in {0,...,2}{
\draw[dashed,opacity=0.5] (\i,0) -- (\i,-3);}
\foreach \j in {0,...,3}{
\foreach \i in {0,...,2}{
\draw[fill=gray] (\i,-\j) circle (2pt);}}
\draw[fill=white] (2,0) circle (6pt); \draw[fill=white] (0,0) circle (6pt);
\draw[fill=white] (1,-1) circle (6pt); \draw[fill=white] (0,-2) circle (6pt);
\draw[fill=white] (2,-2) circle (6pt); \draw[fill=white] (0,-3) circle (6pt);
\draw (0,0) node {$\scriptstyle 0$};
\draw (2,0) node {$\scriptstyle 2$};
\draw (1,-1) node {$\scriptstyle 4$};
\draw (0,-2) node {$\scriptstyle 6$};
\draw (2,-2) node {$\scriptstyle 8$};
\draw (0,-3) node {$\scriptstyle 9$};
\end{tikzpicture}
\end{center}
Moving a bead up one space is equivalent to reducing one of the elements of
$\beta(\la;n)$ by $t$.
This is, in turn, equivalent to removing a $t$-ribbon 
from $\la$ such that what remains is still a Young diagram
(see for instance \cite[p.~12]{Macdonald95}).
Pushing all beads to the top will give the bead configuration of
$\tcore(\la)$, and this is clearly independent of the order in which the beads
are pushed. It follows that $\tcore(\la)$ is the unique partition obtained
by removing $t$-ribbons (in a valid way) from the diagram of $\la$ until
it is no longer possible to do so.
We note that if removing a ribbon $R$ corresponds to moving a bead from position
$b$ to $b-t$, then $\mathrm{ht}(R)$ is equal to the number of beads lying
at the positions strictly between $b-t$ and $b$.
The $t$-quotient can be obtained from the bead configuration by reading the 
$r$-th runner, bottom-to-top, as a bead configuration with $m_r(\la;n)$ beads.
For our example, this means that $\beta(\tcore(\la);6)=\{6,5,3,2,1,0\}$,
so that $\tcore(\la)=(1,1)$, with the quotient $((1,1),(1),(1))$ computed
similarly.

The above procedure of computing the $t$-core and $t$-quotient actually
encodes a bijection
\begin{align*}
\phi_t:\mathscr{P}&\longrightarrow\mathscr{C}_t\times\mathscr{P}^t \\
\la &\longmapsto\big(\tcore(\la),(\lar{0},\dots,\lar{t-1})\big),
\end{align*}
such that $\abs{\la}=\abs{\tcore(\la)}+t(\abs{\lar{0}}+\cdots+\abs{\lar{t-1}})$.
The arithmetic description of this correspondence was first written down by 
Littlewood \cite{Littlewood51}. 
The idea of removing ribbons from a partition until a unique core is obtained 
goes back to Nakayama \cite{Nakayama40}.
The $t$-quotient of a partition has its origin in the \emph{star diagrams}
of Nakayama, Osima, Robinson and Staal \cite{NO51,Robinson48,Staal50},
which were shown to be equivalent to Littlewood's $t$-quotient by Farahat 
\cite{Farahat53}.

Let $w_t(\la;n)$ be the permutation of $\beta(\la;n)$ which sorts the elements
of the beta set so that their residues modulo $t$ are increasing, and the
elements within each residue class decrease.
The sign of $w_t(\la;n)$ will be denoted $\sgn(w_t(\la;n))$.
The permutation $w_t(\la;n)$ can also be read off the bead configuration
by first labelling the beads ``backwards'': label the bead with largest place 
$1$, second-largest $2$, 
and so on. Reading the labels column-wise from bottom-to-top gives 
$w_t(\la;n)$ in one-line notation.
An inversion in this permutation corresponds to a pair of beads $b_1,b_2$ such 
that $b_2$ lies weakly below and strictly to the right of $b_1$.
With the same example as before $w_3((4,4,3,2,1);6)=136425$ and 
the bead at position $0$ generates three inversions, as it ``sees''
the beads $2, 4$ and $8$.

As follows from the above, a partition $\la$ has empty $t$-core if and only
if it is $t$-tileable.
In our results below we will need the following characterisation of when
a skew shape is $t$-tileable, generalising the notion of ``empty $t$-core''
to this case.
We briefly recall our convention that $t$-quotients are always computed
with the number of beads in the bead configuration a multiple of $t$.
\begin{lemma}\label{Lem_tilable}
A skew shape $\la/\mu$ is tileable by $t$-ribbons if and only if 
$\tcore(\la)=\tcore(\mu)$ and $\mur{r}\subseteq\lar{r}$ for each 
$0\leq r\leq t-1$.
\end{lemma}
\begin{proof}
The skew shape being $t$-tileable is equivalent to the diagram of
$\mu$ being obtainable from the diagram of $\la$ by removing $t$-ribbons.
In other words, we can obtain the bead configuration of $\mu$ from that of 
$\la$, where both have $nt$ beads, by moving beads upwards.
Assume that this is the case.
Then $m_r(\la;nt)=m_r(\mu;nt)$ for each $0\leq r\leq t-1$, so that the $r$-th
runner has the same number of beads in each diagram.
This implies that $\tcore(\la)=\tcore(\mu)$.
It also follows that the $i$-th bead in each runner of $\la$'s
bead configuration must lie weakly below the $i$-th bead in the same
runner of $\mu$'s bead configuration.
Equivalently, $\mur{r}_i\leq \lar{r}_i$ for all $0\leq r\leq t-1$ and
$1\leq i\leq m_r(\la;nt)$, which in turn is equivalent to 
$\mur{r}\subseteq\lar{r}$.
The reverse direction is now clear.
\end{proof}
Note that the lemma is also true when the $t$-quotients of $\la$ and
$\mu$ are computed using the same integer $n$ of any residue class modulo $t$.
If $\la/\mu$ is $t$-tileable, then we think of 
$\lar{0}/\mur{0},\dots,\lar{t-1}/\mur{t-1}$ as its $t$-quotient.
When $\la/\mu$ is not $t$-tileable, it is not so clear how to define
the $t$-quotient.

\subsection{Symmetric functions and universal characters}
Here we discuss some basics of the theory of symmetric functions,
following \cite{Macdonald95}.
Let $\La$ denote the \emph{ring of symmetric functions} in an 
arbitrary countable set of variables $X=(x_1,x_2,x_3,\dots)$, called an 
\emph{alphabet}.
Where possible, we write  elements of $\La$ without reference to an alphabet if
the expression is independent of the chosen alphabet.
If for a positive integer $n$ one sets $x_i=0$ for all $i>n$ then the 
elements of
$\La$ reduce to symmetric polynomials in the variables $(x_1,\dots,x_n)$.
Another common specialisation sets 
$x_{n+i}=x_i^{-1}$ for $1\leq i\leq n$ and  $x_i=0$ for $i>2n$.
This gives Laurent polynomials in the $x_i$ invariant under permutation and 
inversion of the variables (i.e., $\mathrm{BC}_n$-symmetric functions).
We will later write $(x_1^\pm,\dots,x_n^\pm)$ for this alphabet.

Two fundamental algebraic bases for $\La$ are the \emph{complete 
homogeneous symmetric functions} and the \emph{elementary symmetric functions},
defined for any positive integer $r$ by
\[
h_r(X):=\sum_{1\leq i_1\leq\cdots\leq i_r}x_{i_1}\cdots x_{i_r}
\quad\text{and}\quad
e_r(X):=\sum_{1\leq i_1<\cdots<i_r}x_{i_1}\cdots x_{i_r},
\]
respectively. We further set $h_0=e_0:=1$ and $h_{-r}=e_{-r}=0$ for 
positive $r$.
These admit the generating functions
\begin{align*}
H_z(X)&:=\sum_{r\geq0}z^rh_r(X)=\prod_{i\geq1}\frac{1}{1-zx_i}\\
E_z(X)&:=\sum_{r\geq0}z^re_r(X)=\prod_{i\geq1}(1+zx_i).
\end{align*}
The $h_r$ and $e_r$ for $r\geq 1$ are algebraically independent
over $\mathbb{Z}$ and generate $\La$.
In view of this, we can define a homomorphism $\omega:\La\longrightarrow\La$
by $\omega h_r=e_r$.
It then follows from the relation $H_z(X)E_{-z}(X)=1$ that
$\omega e_r=h_r$, so that $\omega$ is an involution.
We also define the power sums by
\[
p_r(X):=\sum_{i\geq1} x_i^r,
\]
for $r\geq1$ and $p_0:=1$.
These satisfy $\omega p_r=(-1)^{r-1}p_r$.

The most important family of symmetric functions are the \emph{Schur functions}.
These have several definitions, but
for our purposes it is best to define them, already for skew shapes, by the 
\emph{Jacobi--Trudi formula}.
If $\la/\mu$ is a skew shape and $n$ an integer such that 
$n\geq l(\la)$ we define
\begin{equation}\label{Eq_JT}
s_{\la/\mu}:=\det_{1\leq i,j\leq n}(h_{\la_i-\mu_j-i+j}).
\end{equation}
This is independent of $n$ as long as $n\geq l(\la)$. 
If $\mu\not\subseteq\la$ then we set $s_{\la/\mu}:=0$.
There is also an equivalent formula in terms of the $e_r$,
called the \emph{dual Jacobi--Trudi formula} (rarely also the
\emph{N\"agelsbach--Kostka identity})
\[
s_{\la/\mu}=\det_{1\leq i,j\leq m}(e_{\la_i'-\mu_j'-i+j}).
\]
Restricting to the $\mu$ empty case, 
we have $s_{(r)}=h_r$ and $s_{(1^r)}=e_r$.
Moreover, it is clear that $\omega s_{\la/\mu}=s_{\la'/\mu'}$.

If the set of variables $(x_1,\dots,x_n)$ is finite then the
Schur function for $\mu=0$ admits another definition as a ratio of alternants
\begin{equation}\label{Eq_s-alt}
s_\la(x_1,\dots,x_n)=\frac{\det_{1\leq i,j\leq n}(x_i^{\la_j+n-j})}
{\det_{1\leq i,j\leq n}(x_i^{n-j})}.
\end{equation}
The denominator is the Vandermonde determinant and has the product 
representation 
$\det_{1\leq i,j\leq n}(x_i^{n-j})=\prod_{1\leq i<j\leq n}(x_i-x_j)$.
In this case we also define $s_\la(x_1,\dots,x_n)=0$ if $l(\la)>n$.
If $\la$ is a partition of length at most $n$, then
\begin{equation}\label{Eq_det-rep}
s_{(\la_1+1,\dots,\la_n+1)}(x_1,\dots,x_n)
=(x_1\cdots x_n)s_{(\la_1,\dots,\la_n)}(x_1,\dots,x_n).
\end{equation}
This allows for Schur functions with a finite set of $n$ variables to be 
extended to weakly decreasing sequences of integers of length exactly $n$.

Following Koike and Terada we define the \emph{universal characters} for  
$\mathrm{O}(2n,\mathbb{C})$ and $\mathrm{Sp}(2n,\mathbb{C})$ as 
the symmetric functions \cite[Definition~2.1.1]{KT87}
\begin{align}\label{Eq_oh}
\o_\la&:=\det_{1\leq i,j\leq n}\big(h_{\la_i-i+j}-h_{\la_i-i-j}\big)\\
\sp_\la&:=\frac{1}{2}
\det_{1\leq i,j\leq n}\big(h_{\la_i-i+j}+h_{\la_i-i-j+2}\big),\label{Eq_sph}
\end{align}
where $n\geq l(\la)$. 
Like the Schur functions, these determinants also have dual versions
\begin{align}
\o_\la&=\frac{1}{2}\notag
\det_{1\leq i,j\leq m}\big(e_{\la_i'-i+j}+e_{\la_i'-i-j+2}\big)\\
\sp_\la&=
\det_{1\leq i,j\leq m}\big(e_{\la_i'-i+j}-e_{\la_i'-i-j}\big), \label{Eq_sp-e}
\end{align}
where $m\geq \la_1$.
From this it is clear that $\omega \o_\la=\sp_{\la'}$.
Koike alone added a third universal character for the
group $\mathrm{SO}(2n+1,\mathbb{C})$ \cite[Definition~6.4]{Koike97} 
(see also \cite[Equation~(3.8)]{LRW20})
\[
\so_\la:=\det_{1\leq i,j\leq n}\big(h_{\la_i-i+j}+h_{\la_i-i-j+1}\big)
=\det_{1\leq i,j\leq m}\big(e_{\la_i'-i+j}+e_{\la_i'-i-j+1}\big).
\]
This universal character is self-dual under $\omega$, so 
$\omega \so_\la=\so_{\la'}$.
For later convenience we also define a variant of the above as
\[
\so_\la^{-}:=\det_{1\leq i,j\leq n}\big(h_{\la_i-i+j}-h_{\la_i-i-j+1}\big)
=\det_{1\leq i,j\leq m}\big(e_{\la_i'-i+j}-e_{\la_i'-i-j+1}\big).
\]
If $X=(x_1,x_2,x_3,\dots)$ is a set of variables (which may be finite or
countable)
and $-X:=(-x_1,-x_2,-x_3,\dots)$, then
\[
\so_\la^{-}(X)=(-1)^{\abs{\la}}\so_\la(-X),
\]
since the $h_r$ and $e_r$ are homogeneous of degree $r$.

For $l(\la)\leq n$, each of the three above universal characters become actual 
characters of irreducible representations of their associated groups
when specialised to $(x_1^\pm,\dots,x_n^\pm)$ (hence the name
universal characters).

The irreducible polynomial representations of $\mathrm{GL}(n,\mathbb{C})$
are indexed by partitions of length at most $n$.
On the other hand, the irreducible rational representations are indexed by
weakly decreasing sequences of integers of length $n$,
which are called \emph{staircases} by Stembridge \cite{Stembridge87}.
Such sequences are equivalent to pairs of partitions $\la,\mu$ 
such that $l(\la)+l(\mu)\leq n$.
Given such a pair, one defines the associated staircase $[\la,\mu]$ by
$[\la,\mu]_i:=\la_i-\mu_{n-i+1}$ for $1\leq i\leq n$.
The characters of the rational representations of $\mathrm{GL}(n,\mathbb{C})$
are then given by $s_{[\la,\mu]}(x_1,\dots,x_n)$ for all staircases with
$n$ entries.
Note that \eqref{Eq_det-rep} implies that this object is just a Schur function
up to a power of $x_1\cdots x_n$.
In \cite{Littlewood44}, Littlewood gave the expansion
\begin{equation}\label{Eq_Littlewood-def}
s_{[\la,\mu]}(x_1,\dots,x_n)
=\sum_{\nu}(-1)^{\abs{\nu}}s_{\la/\nu}(x_1,\dots,x_n)
s_{\mu/\nu'}(1/x_1,\dots,1/x_n).
\end{equation}
For a pair of partitions $\la,\mu$ and sets of indeterminates $X$, $Y$,
this may be used to define the universal character associated to a rational
representation of $\mathrm{GL}(n,\mathbb{C})$ as
\begin{equation}\label{Eq_rs-def}
\rs_{\la,\mu}(X;Y)
:=\sum_{\nu}(-1)^{\abs{\nu}}s_{\la/\nu}(X)s_{\mu/\nu'}(Y).
\end{equation}
Note that the only terms which contribute are those with $\nu\subseteq\la$ 
and $\nu'\subseteq\mu$.
If we let $\omega_X$ and $\omega_Y$ denote the involution $\omega$ acting on
the set of variables in its subscript, then
\begin{align*}
\omega_X\omega_Y\rs_{\la,\mu}(X;Y)
&=\sum_{\nu}(-1)^{\abs{\nu}}s_{\la'/\nu'}(X)s_{\mu'/\nu}(Y)\\
&=\sum_{\nu'}(-1)^{\abs{\nu}}s_{\la'/\nu}(X)s_{\mu'/\nu'}(Y)\\
&=\rs_{\la',\mu'}(X;Y).
\end{align*}
As shown by Koike \cite{Koike89}, this object has a Jacobi--Trudi-type 
expression as a block matrix
\begin{equation}\label{Eq_rs-h}
\rs_{\la,\mu}(X;Y)=
\det_{1\leq i,j\leq n+m}
\begin{pmatrix}
\big(h_{\la_i-i+j}(X)\big)_{1\leq i,j\leq n} & 
\big(h_{\la_i-i-j+1}(X)\big)_{\substack{1\leq i\leq n\\ 1\leq j\leq m}} \\
\big(h_{\mu_i-i-j+1}(Y)\big)_{\substack{1\leq i\leq m\\ 1\leq j\leq n}} &
\big(h_{\mu_i-i+j}(Y)\big)_{1\leq i,j\leq m}
\end{pmatrix},
\end{equation}
where $n\geq l(\la)$ and $m\geq l(\mu)$.
As for the other determinants, this is independent of $n$ and $m$ as long as
$n\geq l(\la)$ and $m\geq l(\mu)$.
The relation under $\omega_X\omega_Y$ implies we have the dual 
form \cite[Definition~2.1]{Koike89}
\begin{equation}\label{Eq_rs-e}
\rs_{\la,\mu}(X;Y)=
\det_{1\leq i,j\leq n+m}
\begin{pmatrix}
\big(e_{\la_i'-i+j}(X)\big)_{1\leq i,j\leq n} & 
\big(e_{\la_i'-i-j+1}(X)\big)_{\substack{1\leq i\leq n\\ 1\leq j\leq m}} \\
\big(e_{\mu_i'-i-j+1}(Y)\big)_{\substack{1\leq i\leq m\\ 1\leq j\leq n}} &
\big(e_{\mu_i'-i+j}(Y)\big)_{1\leq i,j\leq m}
\end{pmatrix},
\end{equation}
where $n\geq \la_1$ and $m\geq\mu_1$.
The definition \eqref{Eq_rs-def} and the determinants \eqref{Eq_rs-h} 
and \eqref{Eq_rs-e} are related by taking the Laplace expansion of 
each determinant according to its presented block structure;
see \cite[Equation~(2.1)]{Koike89}.
Also, by the definition \eqref{Eq_rs-def} and \eqref{Eq_Littlewood-def} it is
immediate that, for $l(\la)+l(\mu)\leq n$,
\[
\rs_{\la,\mu}(x_1,\dots,x_n;1/x_1,\dots,1/x_n)=s_{[\la,\mu]}(x_1,\dots,x_n).
\]
We will always take $X=Y$ in $\rs_{\la,\mu}(X;Y)$, which we write as 
$\rs_{\la,\mu}$ in the rest of the paper.
In particular we note that $\rs_{\la,\mu}=\rs_{\mu,\la}$.

As mentioned in the introduction, the notion of twisting the set of variables
$x_1,\dots,x_n$ by a primitive $t$-th root of unity $\zeta$ is replaced by
the operator $\varphi_t$ \eqref{Eq_phi-def}, which has been considered
by Macdonald \cite[p.~91]{Macdonald95} and, for $t=2$, by Baik and Rains
\cite[p.~25]{BR01}.
Let $X^t:=(x_1^t,x_2^t,x_3^t\dots)$ and denote by $\psi_t$ the
homomorphism
\begin{align*}
\psi_t: \La&\longrightarrow \La_{X^t}\\
f&\longmapsto f(X,\zeta X,\dots,\zeta^{t-1}X).
\end{align*}
Since $\psi_t H_z(X)=H_{z^t}(X^t)$, both $\varphi_t$ and $\psi_t$ act
on the $h_r$ in the same way, i.e., the diagram
\begin{equation}\label{Eq_diagram}
\begin{tikzcd}
\La\arrow{r}{\varphi_t} \arrow{dr}{\psi_t} & \La \arrow{d}{X^t}\\
& \La_{X^t}
\end{tikzcd}
\end{equation}
commutes, where the arrow labelled $X^t$ is the substitution map.
This implies the claim of the introduction that the action of $\varphi_t$ is
equivalent to twisting the alphabet $X$ by a primitive $t$-th root of unity
$\zeta$.
If one wishes to think about this as a map $\La_X\longrightarrow\La_X$
where $X$ is some concrete alphabet, then substitute each $x\in X$ by its set 
of $t$-th roots $x^{1/t},\zeta x^{1/t},\dots,\zeta^{t-1}x^{1/t}$ and
evaluate this expression.
By the action of this map on the $h_r$, such a map gives a symmetric function
again in the variables $X$.
Using the generating function $E_z(X)$ one may also show that
\cite[\S5.8]{Lascoux?}
\[
\varphi_t e_r=\begin{cases}
(-1)^{r(t-1)/t}e_{r/t} & \text{if $t$ divides $r$}, \\
0 & \text{otherwise}.
\end{cases}
\]
Therefore $\omega$ and $\varphi_t$ commute if $t$ is odd, but not in general.
Proposition~\ref{Prop_signs} shows that, in some cases, the maps commute up
to a computable sign.
Different, but closely related operators are discussed at the end of this paper.
 
\section{Summary of results}\label{Sec_results}
With the preliminary material of the previous section under our belts, we
are now ready to state our main results regarding factorisations of 
universal characters under $\varphi_t$. 
The first of these is the action of the map on the skew Schur functions.

\begin{theorem}\label{Thm_skew}
We have that $\varphi_t s_{\la/\mu}=0$ unless $\la/\mu$ is tileable by 
$t$-ribbons, in which case
\[
\varphi_t s_{\la/\mu}=\sgn_t(\la/\mu)
\prod_{r=0}^{t-1}s_{\lar{r}/\mur{r}}.
\]
\end{theorem}
For $\mu$ empty, this result is due to Littlewood
\cite[p.~131]{Littlewood40}, who proves it by direct manipulation of the 
ratio of alternants \eqref{Eq_s-alt}.
By \eqref{Eq_diagram} with $X=(x_1,\dots,x_n)$, one can recover Littlewood's 
result by simply evaluating the right-hand side of the equation
at $(x_1^t,\dots,x_n^t)$.
He also states the $\mu$ empty case of the theorem 
in the language of symmetric group characters:
see both \cite[p.~144]{Littlewood40} and \cite[p.~340]{Littlewood51}.
The generalisation to skew characters was 
discovered by Farahat \cite{Farahat54} (see also \cite[Theorem~3.3]{EPW14}).
The form we state here is precisely that of Macdonald
\cite[p.~91]{Macdonald95}.
Curiously, Prasad recently rediscovered the $\mu$ empty case
independently, with a proof identical to Littlewood's, but
in a more representation-theoretic context \cite{Prasad16}.
A version of the result for Schur's $P$-and $Q$-functions has been given by
Mizukawa \cite[Theorem~5.1]{Mizukawa02}.

Theorem~\ref{Thm_skew} has been rediscovered many times for 
both skew and straight shapes, and often only in special cases.
We make no attempt to give a complete history, but it appears to us that 
the theorem deserves to be better known.
The interested reader can consult \cite{Wildon22} for some exposition on 
the character theory side of this story. 
On the symmetric functions side, such an exposition is lacking in the
literature.

We now state, in sequence, the three factorisations lifting
\cite[Theorems~2.11, 2.15 \& 2.17]{AK22} to the level of universal characters,
beginning with the universal orthogonal character.
\begin{theorem}\label{Thm_AK_o}
Let $\la$ be a partition of length at most $nt$. 
Then $\varphi_t\o_\la=0$ unless $\tcore(\la)$ is orthogonal, in which case
\[
\varphi_t\o_\la=
(-1)^{\varepsilon^{\o}_{\la;nt}} \sgn(w_t(\la;nt))
\o_{\lar{0}}\prod_{r=1}^{\lfloor(t-1)/2\rfloor}\rs_{\lar{r},\lar{t-r}}
\times\begin{cases}
\so_{\lar{t/2}}^- & \text{$t$ even}, \\
1 & \text{$t$ odd},
\end{cases}
\]
where
\[
\varepsilon^{\o}_{\la;nt}
=\sum_{r=\lfloor(t+2)/2\rfloor}^{t-1}\binom{m_r(\la;nt)+1}{2}
+\rk(\tcore(\la))
+\begin{cases}
\binom{n+1}{2}+n\rk(\tcore(\la)) & \text{$t$ even}, \\
0 & \text{$t$ odd}.
\end{cases}
\]
\end{theorem}
Our next result is the same factorisation for the symplectic character.
\begin{theorem}\label{Thm_AK_sp}
Let $\la$ be a partition of length at most $nt$. 
Then $\varphi_t\sp_\la=0$ unless $\tcore(\la)$ is symplectic, in which case
\[
\varphi_t\sp_\la=
(-1)^{\varepsilon^{\sp}_{\la;nt}} \sgn(w_t(\la;nt))\sp_{\lar{t-1}}
\prod_{r=0}^{\lfloor(t-3)/2\rfloor}\rs_{\lar{r},\lar{t-r-2}}
\times\begin{cases}
\so_{\lar{(t-2)/2}} & \text{$t$ even}, \\
1 & \text{$t$ odd},
\end{cases}
\]
where
\[
\varepsilon^{\sp}_{\la;nt}
=\sum_{r=\lfloor t/2\rfloor}^{t-2}\binom{m_r(\la;nt)+1}{2}
+\begin{cases} \binom{n+1}{2}+
n\rk(\tcore(\la)) & \text{$t$ even}, \\
0 & \text{$t$ odd}.
\end{cases}
\]
\end{theorem}
Finally, we can claim a similar factorisation for $\so_\la$.
\begin{theorem}\label{Thm_AK_so}
Let $\la$ be a partition of length at most $nt$. 
Then $\varphi_t\so_\la=0$ unless $\tcore(\la)$ is self-conjugate, in which case
\[
\varphi_t\so_\la=
(-1)^{\varepsilon^{\so}_{\la;nt}} \sgn(w_t(\la;nt))
\prod_{r=0}^{\lfloor(t-2)/2\rfloor}\rs_{\lar{r},\lar{t-r-1}}
\times\begin{cases}
1 & \text{$t$ even},\\
\so_{\lar{(t-1)/2}} & \text{$t$ odd},
\end{cases}
\]
where\footnote{We have corrected the lower bound in the sum defining 
$\varepsilon^{\so}_{\la;nt}$ from $\lfloor t/2\rfloor$ in 
\cite[Theorem~2.17]{AK22} (there denoted $\epsilon$) to 
$\lfloor (t+1)/2\rfloor$.}
\[
\varepsilon^{\so}_{\la;nt}=
\sum_{r=\lfloor (t+1)/2\rfloor}^{t-1}\binom{m_r(\la;nt)+1}{2}+
\begin{cases}
0 & \text{$t$ even}, \\
n\rk(\tcore(\la)) & \text{$t$ odd}.
\end{cases}
\]
\end{theorem}

Some remarks are in order. 
Firstly, the three signs $\sgn(w_t(\la;nt))(-1)^{\varepsilon_{\la;nt}^\bullet}$ 
are actually independent of $n$ as long as $nt\geq l(\la)$, a fact which 
we prove in Lemma~\ref{Lem_sign-invariance} below.
As remarked by Ayyer and Kumari \cite[Remark~2.19]{AK22}, 
the order of the quotient is unchanged upon replacing $n\mapsto n+1$,
so the product in the evaluation is independent of $n$.
It is in principle possible to carry out our proof technique below under
the assumption that $l(\la)$ is bounded by an arbitrary integer, say $k$, 
where $k$ is not necessarily a multiple of $t$.
In this case the evaluation is of course the same, however the sign will be
expressed differently and the $t$-quotients in the evaluations will be 
a cyclic permutation of the ones presented.
Since the proof is simplest when this $k$ is a multiple of $t$, 
we stick to this case.

To obtain the theorems of Ayyer and Kumari one evaluates the right-hand side
of each identity at the set of variables $(x_1^{\pm t},\dots,x_n^{\pm t})$.
Using \eqref{Eq_det-rep} and the definition of $\rs_{\la,\mu}$ it follows
that in this case the rational universal characters occurring in each evaluation
agree with the Schur functions 
$s_{\mur{k}_i}(x_1^{\pm t},\dots,x_n^{\pm t})$ in the notation of
\cite{AK22}.

As we have already seen the maps $\omega$ and $\varphi_t$ do not commute in
general.
However, when acting on $s_{\la/\mu}$ and $\so_\la$, they commute up to an
explicitly computable sign.
\begin{prop}\label{Prop_signs}
We have the relations
\[
\omega\varphi_t s_{\la/\mu}=
(-1)^{(t-1)(\abs{\lar{0}/\mur{0}}+\cdots+\abs{\lar{t-1}/\mur{t-1}})}
\varphi_t\omega s_{\la/\mu},
\]
and
\[
\omega\varphi_t \so_\la =
(-1)^{(t-1)(\abs{\lar{0}}+\cdots+\abs{\lar{t-1}})}\varphi_t\omega \so_\la.
\]
\end{prop}

We remark that the second relation does not hold with $\so_\la$ replaced 
by $\sp_\la$ or $\o_\la$
as written above since $\omega\so_\la^-=\so_{\la'}^-$.

\section{Auxiliary results}\label{Sec_lemmas}
The purpose of this section is to collect all the small facts about 
beta sets and the signs \eqref{Eq_sgn-def} which we need to prove our
main results.
To begin, we relate the bead configurations of a partition and its conjugate.
\begin{lemma}\label{Lem_reverse}
Let $\la$ be a partition of length at most $nt$ such that $\la_1\leq mt$.
Then the bead configuration for $\beta(\la';mt)$ can be obtained from
the bead configuration for $\beta(\la;nt)$ with $n+m$ rows by rotating 
the picture by $180^\circ$ and then interchanging beads and spaces.
\end{lemma}
\begin{proof}
This is a consequence of the fact \cite[p.~3]{Macdonald95} that
for $l(\la)\leq n$ and $\la_1\leq m$,
\[
\{0,1,\dots,m+n-1\}=\{\la_i+n-i:1\leq i\leq n\}\sqcup
\{m+n-1-(\la_j'+m-j):1\leq j\leq m\},
\]
where $\sqcup$ denotes a disjoint union.
\end{proof}
This lemma immediately implies the following relationship between the
$t$-core and $t$-quotient of $\la$ and $\la'$.
\begin{coro}\label{Cor_reverse}
For a partition $\la$ we have $\tcore(\la')=\tcore(\la)'$ and the $t$-quotient
of $\la'$ is $\big((\lar{t-1})',\dots,(\lar{0})'\big)$.
\end{coro}

The next pair of lemmas are due to Ayyer and Kumari, the first of which
characterises partitions with $z$-asymmetric $t$-cores in terms of their
beta sets \cite[Lemma~3.6]{AK22}.
\begin{lemma}\label{Lem_AK_orth}
For a partition $\la$ of length at most $nt$, $\tcore(\la)$ is of the form 
$(a\mid a+z)$ for some integer $-1\leq z\leq t-1$ if and only if
\begin{subequations}\label{Eq_m_conditions}
\begin{align}
m_r(\la,nt)+m_{t-r-z-1}(\la,nt)&=2n \quad \text{for $0\leq r\leq t-z-1$}, \\
m_r(\la,nt)&=n \quad \text{for $t-z\leq r\leq t-1$},
\end{align}
\end{subequations}
where the indices of the $m_r$ are taken modulo $t$.
\end{lemma}
The second lemma of Ayyer and Kumari we need is \cite[Lemma~3.13]{AK22},
which is used later on to simplify signs.
\begin{lemma}
Let $\la$ be a partition of length at most $nt$.
If $\tcore(\la)$ is orthogonal, then
\begin{align}\label{Eq_o-rk}
\rk(\tcore(\la))=\sum_{r=1}^{\lfloor(t-1)/2\rfloor}\abs{m_r(\la;nt)-n}
&=\sum_{r=\lfloor(t+2)/2\rfloor}^{t-1}\abs{m_r(\la;nt)-n}\\
\intertext{If $\tcore(\la)$ is symplectic, then}
\label{Eq_sp-rk}
\rk(\tcore(\la))=\sum_{r=0}^{\lfloor(t-3)/2\rfloor}\abs{m_r(\la;nt)-n}
&=\sum_{r=\lfloor t/2\rfloor}^{t-2}\abs{m_r(\la;nt)-n}\\
\intertext{If $\tcore(\la)$ is self-conjugate, then}
\label{Eq_so-rk}
\rk(\tcore(\la))=\sum_{r=0}^{\lfloor(t-2)/2\rfloor}\abs{m_r(\la;nt)-n}
&=\sum_{r=\lfloor(t+1)/2\rfloor}^{t-1}\abs{m_r(\la;nt)-n}
\end{align}
\end{lemma}

Next, we show that the sign of a tileable skew shape can be expressed in 
terms of the signs of the permutations $w_t(\la;n)$.
\begin{lemma}\label{Lem_signs0}
For $\la/\mu$ $t$-tileable and any integer $n$ such that $n\geq l(\la)$,
\[
\sgn_t(\la/\mu)=\sgn(w_t(\la;n))\sgn(w_t(\mu;n)).
\]
\end{lemma}
\begin{proof}
Since $\la/\mu$ is $t$-tileable, it has a ribbon decomposition
$D=(\nur{0},\dots,\nur{k})$ where $\nur{0}=\mu$ and $\nur{k}=\la$.
Also, $\nur{k-1}$ can be obtained from $\la$ by moving one bead at some
position upward one space.
By our characterisation of the inversions in the permutation $w_t(\la;n)$,
we see that moving a bead at position $\ell$ up one space changes the sign by 
$(-1)^{b_k}$ where $b_k$ is the number of beads at positions between
$\ell-t$ and $\ell$.
In other words, $\sgn(w_t(\la;n))=(-1)^{b_k}\sgn(w_t(\nur{k-1};n))$.
Moreover, $b_k=\mathrm{ht}(\nur{k}/\nur{k-1})$, so that
\[
\sgn(w_t(\la;n))\sgn(w_t(\mu;n))=(-1)^{\sum_{i=1}^k b_i}=
(-1)^{\mathrm{ht}(D)}=\sgn_t(\la/\mu). \qedhere
\]
\end{proof}
We also have the following useful relationship between the sign of $\la/\mu$ 
and $\la'/\mu'$.
\begin{lemma}\label{Lem_signs1}
For $\la/\mu$ $t$-tileable,
\[
\sgn_t(\la/\mu)\sgn_t(\la'/\mu')=
(-1)^{(t-1)(\abs{\lar{0}}+\cdots+\abs{\lar{t-1}}-
\abs{\mur{0}}-\cdots-\abs{\mur{t-1}})}.
\]
\end{lemma}
\begin{proof}
To prove the claim of the lemma we will proceed by induction on $\abs{\la/\mu}$.
If $\abs{\la/\mu}=0$ then $\la=\mu$ and the equation is trivial.
Now fix $\mu$ and assume the result holds for $\la/\mu$ being $t$-tileable.
Adding a $t$-ribbon to $\la/\mu$ moves one of the beads, say at position $b$,
in the bead configuration for $\la$ down a single space.
The change in the number of inversions in $w_t(\la;nt)$ is the number of beads 
$b'$ such that $b<b'<b+1$.
A consequence of Lemma~\ref{Lem_reverse} is that $w_t(\la';mt)$ will change
by the number of empty spaces between $b$ and $b+1$.
There are $t-1$ spaces and beads between $b$ and $b+1$, so the left-hand
side changes by $(-1)^{t-1}$ when adding a $t$-ribbon.
But adding a $t$-ribbon to $\la/\mu$ changes some element of the $t$-quotient 
of $\la$ by a single box, also corresponding to a change in sign of 
$(-1)^{t-1}$.
\end{proof}
There is another sign relation between orthogonal and symplectic $t$-cores,
but this time using the permutations $w_t$.
\begin{lemma}\label{Lem_signs2}
Let $\la$ be an orthogonal or symplectic $t$-core whose diagram is contained 
in an $nt\times nt$ square. Then
\[
\sgn(w_t(\la;nt))\sgn(w_t(\la';nt))=(-1)^{\rk(\la)}.
\]
\end{lemma}
\begin{proof}
Assume that $\la$ is a non-empty, orthogonal $t$-core (if $\la$ is empty
the result is trivial) and fix $n$ so that the condition of the theorem holds.
The key observation is that for an orthogonal $t$-core, the bead configuration
of $\la'$ with $nt$ beads can be obtained from the bead configuration of $\la$
with $nt$ beads by reducing the labels by $1$ modulo $t$.
For example if $\la=(12,7,5,3,2,2,1,1,1,1,1)$ then
$\la'=(11,6,4,3,3,2,2,1,1,1,1,1)$ and their bead configurations
for $t=6$ and $n=2$ are
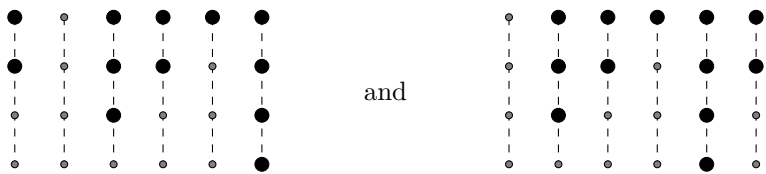
\begin{center}
\begin{tikzpicture}[scale=0.65]
\foreach \i in {0,...,5}{\draw[dashed,opacity=0.5] (\i,0) -- (\i,-3);}
\foreach \i in {0,...,5}{\draw[dashed,opacity=0.5] (\i+10,0) -- (\i+10,-3);}
\foreach \j in {0,...,3}{  
\foreach \i in {0,...,5}{
\draw[fill=gray] (\i,-\j) circle (2pt);}}
\filldraw[black] (0,0) circle (4pt); \filldraw[black] (0,-1) circle (4pt);
\filldraw[black] (2,0) circle (4pt); \filldraw[black] (2,-1) circle (4pt);
\filldraw[black] (2,-2) circle (4pt);\filldraw[black] (3,0) circle (4pt);
\filldraw[black] (3,-1) circle (4pt);\filldraw[black] (4,0) circle (4pt);
\filldraw[black] (5,0) circle (4pt);\filldraw[black] (5,-1) circle (4pt);
\filldraw[black] (5,-2) circle (4pt);\filldraw[black] (5,-3) circle (4pt);
\foreach \j in {0,...,3}{  
\foreach \i in {0,...,5}{
\draw[fill=gray] (\i+10,-\j) circle (2pt);}}
\filldraw[black] (15,0) circle (4pt); \filldraw[black] (15,-1) circle (4pt);
\filldraw[black] (11,0) circle (4pt); \filldraw[black] (11,-1) circle (4pt);
\filldraw[black] (11,-2) circle (4pt);
\filldraw[black] (12,0) circle (4pt);\filldraw[black] (12,-1) circle (4pt);
\filldraw[black] (13,0) circle (4pt);\filldraw[black] (14,0) circle (4pt);
\filldraw[black] (14,-1) circle (4pt);\filldraw[black] (14,-2) circle (4pt);
\filldraw[black] (14,-3) circle (4pt);
\node at (7.5,-1.5) {and};
\end{tikzpicture}
\end{center}
respectively, where we have suppressed the labels.
This is a consequence of Lemma~\ref{Lem_AK_orth} with $z=\pm1$ and
Lemma~\ref{Lem_reverse}.
When passing from $\la$ to $\la'$, the inversions contributed by the beads
in the first runner are removed and replaced by additional inversions associated
to the remaining beads in the first $n$ rows.
Modulo two, this is equivalent to each bead in the zeroth runner now seeing 
all of the beads in the same row twice, plus all other beads in the other 
runners once.
Let $b$ be the number of beads in the first $n$ rows of the runners
from $1$ to $t-1$ in the bead configuration of $\la$.
Then the sign change is
\[
\sgn(w_t(\la;nt))=\sgn(w_t(\la';nt))(-1)^{n^2(t-1)+b}.
\]
Since $\la$ is an orthogonal $t$-core, $n^2(t-1)+b\equiv \rk(\la)\Mod{2}$ by
\eqref{Eq_o-rk} and Lemma~\ref{Lem_AK_orth} with $z=-1$.
\end{proof}
The next lemma proves the claim made after  
Theorems~\ref{Thm_AK_o}--\ref{Thm_AK_so} that the signs occurring in those
factorisations are independent of $n$.
\begin{lemma}\label{Lem_sign-invariance}
The signs $(-1)^{\varepsilon_{\la;nt}^{\bullet}}\sgn(w_t(\la;nt))$ for 
$\bullet\in\{\o,\sp,\so\}$ are independent of $n$ as long as $nt\geq l(\la)$.
\end{lemma}
\begin{proof}
Assume that $nt\geq l(\la)$. Incrementing $n$ by one adds a row of beads to
the top of the bead configuration of $\la$, and so
$m_r(\la;(n+1)t)=m_r(\la;nt)+1$.
In the inversion count, the $r$th bead in the new first row sees
\[
\sum_{k=r+1}^{t-1}(m_{k}(\la;nt)+1)
\]
other beads.
Summing over $k=0,\dots,t-1$ we see that 
\[
\sgn(w_t(\la;(n+1)t))=\sgn(w_t(\la;nt))
(-1)^{\sum_{r=1}^{\lfloor t/2\rfloor}(m_{2r-1}(\la;nt)+1)}.
\]
Now assume that $\la$ has an orthogonal $t$-core.
Then by Lemma~\ref{Lem_AK_orth} with $z=-1$ the above has the same parity as
\[
\sum_{r=1}^{\lfloor t/2\rfloor}(m_{2r-1}(\la;nt)+1)
\equiv\begin{cases}\frac{(n+1)t}{2} & \text{$t$ even,} 
\\[2mm]
\frac{t-1}{2}+\sum_{r=1}^{(t-1)/2}m_r(\la;nt) & \text{$t$ odd},
\end{cases}
\Mod{2}.
\]
A short calculation shows that
\begin{align*}
\varepsilon_{\la;(n+1)t}^{\o}
&=\varepsilon_{\la;nt}^{\o}+\sum_{r=1}^{\lfloor (t-1)/2\rfloor}
m_r(\la;nt)+\begin{cases}
n+\frac{t}{2}+\rk(\tcore(\la)) & \text{$t$ even},\\[1mm]
\frac{t-1}{2} & \text{$t$ odd},
\end{cases} \\
&\equiv\varepsilon_{\la;nt}^{\o}+\begin{cases}
\frac{(n+1)t}{2} & \text{$t$ even},\\[1mm]
\frac{t-1}{2} +
\sum_{r=1}^{(t-1)/2}m_r(\la;nt)
& \text{$t$ odd}
\end{cases}
\Mod{2}.
\end{align*}
where the last equality uses \eqref{Eq_o-rk}.
The remaining two cases follow similarly.
\end{proof}

We conclude this section with a small lemma relating the indices in the 
Jacobi--Trudi determinants with partition quotients.
\begin{lemma}\label{Lem_phi}
Let $\la,\mu$ be partitions of length at most $nt$ and 
assume that for $0\leq r,s\leq t-1$ we have
$\la_i-i\equiv r\Mod{t}$, $\mu_j-j\equiv s\Mod{t}$ for $1\leq i,j\leq nt$.
If $r-s+z\equiv 0\Mod{t}$ for some $z\in\mathbb{Z}$, then
\[
\frac{\la_i-\mu_j+j-i+z}{t}=\lar{r}_k-\mur{s}_\ell-k+\ell
+m_r(\la;nt)-m_s(\mu;nt)+(r-s+z)/t,
\]
for some $k,\ell$ such that 
$1\leq k\leq m_r(\la;nt)$ and $1\leq\ell\leq m_s(\mu;nt)$.
Alternatively, if $r+s+z\equiv 0\Mod{t}$ then
\[
\frac{\la_i+\mu_j-i-j+z}{t}
=\lar{r}_k+\mur{s}_\ell-k-\ell-2n+1+m_r(\la;nt)+m_s(\mu;nt)+(r+s+z)/t
\]
for some $k, \ell$ such that 
$1\leq k\leq m_r(\la;nt)$ and $1\leq\ell\leq m_{s}(\mu;nt)$.
\end{lemma}
\begin{proof}
We first write $\la_i+nt-i=\xir{r}_kt+r$ and $\mu_j+nt-j=\pi_\ell^{(s)}t+s$
for $1\leq k\leq m_r(\la;nt)$ and $1\leq \ell\leq m_s(\mu;nt)$.
Then
\begin{align*}
\frac{\la_i-\mu_j-i+j+z}{t}
&=\xir{r}_k+\pi_\ell^{(s)}+(r-s+z)/t\\
&=\lar{r}_k+\mur{s}_\ell-k+\ell+m_r(\la;nt)-m_s(\mu;nt)+(r-s+z)/t,
\end{align*}
by the definition of the $t$-quotient. The second claim is analogous.
\end{proof}

\section{Proofs of theorems}\label{Sec_proofs}
In this section we provide proofs of Theorems~\ref{Thm_AK_o},~\ref{Thm_AK_sp}
and \ref{Thm_AK_so}. Since our proof strategy follows that of Macdonald's proof
of the skew Schur case \cite[p.~91]{Macdonald95} (Theorem~\ref{Thm_skew} above),
we reproduce this proof in detail as preparation for what remains.
We also give a detailed example in the orthogonal case in 
Section~\ref{Sec_example} to further elucidate the structure of the remaining
proofs.

\subsection{Proof of Theorem~\ref{Thm_skew}}
Let $n$ be a nonnegative integer and $\mu\subseteq\la$ be a pair of partitions 
such that $l(\la)\leq nt$.
Consider the Jacobi--Trudi determinant
\[
s_{\la/\mu}=\det_{1\leq i,j\leq nt}(h_{\la_i-\mu_j-i+j}).
\]
Before applying the map $\varphi_t$, we rearrange the rows and columns of
this determinant by the permutations $w_t(\la;nt)$ and $w_t(\mu;nt)$ 
respectively.
By Lemma~\ref{Lem_signs0} this introduces a sign of $\sgn_t(\la/\mu)$.
The rows and columns are now arranged in such a way that the residue
classes of $\la_i-i$ and $\mu_j-j$ are grouped in ascending order, and 
the values within each class are decreasing. 
From this vantage point it is easy to apply the map $\varphi_t$ since
$\varphi_t h_{\la_i-\mu_j-i+j}$ vanishes unless $\la_i-i\equiv\mu_j-j\Mod{t}$.
Therefore, $\varphi_t s_{\la/\mu}$ has a block-diagonal structure, with
each block having size $m_r(\la;nt)\times m_r(\mu;nt)$ for $0\leq r\leq t-1$.
We conclude that $\varphi_t s_{\la/\mu}=0$ unless $m_r(\la;nt)=m_r(\mu;nt)$
for all $0\leq r\leq t-1$.
Assuming this is the case, then the entries of the of the minor corresponding
to the residue class $r$ are given by Lemma~\ref{Lem_phi}, and are
\[
h_{(\la_i-\mu_j-i+j)/t}=h_{\lar{r}_k-\mur{r}_\ell-k+\ell}
\]
for some $k$ and $\ell$ with $1\leq k,\ell\leq n$.
Note that the rows and columns are in the desired order (i.e., in each 
$n\times n$ minor the indices increase from $1$ to $n$) thanks to the 
permutations we applied at the beginning of the proof.
We have therefore shown that if $m_r(\la;nt)=m_r(\mu;nt)$ for all 
$0\leq r\leq t-1$, then
\[
\varphi_t s_{\la/\mu}=\sgn_t(\la/\mu)\prod_{r=0}^{t-1}s_{\lar{r}/\mur{r}}.
\]
Now, if $\mur{r}\not\subseteq\lar{r}$ for any $r$ such that 
$0\leq r\leq t-1$ this expression will give zero, 
from which we conclude, by Lemma~\ref{Lem_tilable},
that $\varphi_t s_{\la/\mu}=0$ unless $\la/\mu$ is $t$-tileable.

\subsection{An example}\label{Sec_example}
The structure of the remaining proofs is best outlined through a detailed
example.
To this end, let $t=4$ and $\la=(12,12,12,8,8,8,7,7,3,3,2)$.
We therefore have that $\xcore{4}(\la)=(4,1,1)$, which is clearly 
orthogonal, and
\[
\big(\lar{0},\lar{1},\lar{2},\lar{3}\big)
=\big((2,2),(4,1),(3,2,1),(2,1,1)\big).
\]
Now choose $n=3$, so that $nt=12\geq l(\la)$.
Using the definition of $\o_\la$ as a Jacobi--Trudi-type determinant
\eqref{Eq_oh} we immediately see that
\[
\varphi_4\o_\la
=\begin{vmatrix}
h_3 & \cdot & -h_2 & \cdot & h_4 & \cdot & -h_1 & \cdot & h_5 & \cdot & -1 & \cdot \\
\cdot & h_3-h_2 & \cdot & \cdot & \cdot & h_4-h_1 & \cdot & \cdot & \cdot & h_5-1 & \cdot & \cdot\\
-h_2 & \cdot & h_3 & \cdot & -h_1 & \cdot & h_4 & \cdot & -1 & \cdot & h_5 & \cdot \\
\cdot & \cdot & \cdot & h_2-1 & \cdot & \cdot & \cdot & h_3 & \cdot & \cdot & \cdot & h_4\\
h_1 & \cdot & -1 & \cdot & h_2 & \cdot & \cdot & \cdot & h_3 & \cdot & \cdot & \cdot \\
\cdot & h_1-1 & \cdot & \cdot & \cdot & h_2 & \cdot & \cdot & \cdot & h_3 & \cdot & \cdot\\
\cdot & \cdot & \cdot & h_1 & \cdot & \cdot & \cdot & h_2 & \cdot & \cdot & \cdot & h_3\\
1 & \cdot & \cdot & \cdot & h_1 & \cdot & \cdot & \cdot & h_2 & \cdot & \cdot & \cdot \\
\cdot & \cdot & \cdot & \cdot & \cdot & 1 & \cdot & \cdot & \cdot & h_1 & \cdot & \cdot\\
\cdot &\cdot & \cdot & \cdot & \cdot & \cdot & 1 & \cdot & \cdot & \cdot & h_1 & \cdot \\
\cdot &\cdot & \cdot & \cdot & \cdot & \cdot & \cdot & \cdot & \cdot & 1 & \cdot & \cdot \\
\cdot &\cdot & \cdot & \cdot & \cdot & \cdot & \cdot & \cdot & \cdot & \cdot & \cdot & 1\\
\end{vmatrix},
\]
where we write $\cdot$ in place of $0$ to avoid clutter.
The next step is to permute the rows and columns of the matrix according to
the permutations $w_4(\la;12)$ and $w_4(0;12)$, respectively.
In this case, the first permutation is odd and the second even, 
so we are left with
\[
\varphi_4\o_\la
=-\begin{vmatrix}
h_2-1 & h_3 & h_4 & \cdot & \cdot & \cdot & \cdot & \cdot & \cdot & \cdot & \cdot & \cdot \\
h_1 & h_2 & h_3 & \cdot & \cdot & \cdot & \cdot & \cdot & \cdot & \cdot & \cdot & \cdot \\
\cdot & \cdot & 1 & \cdot & \cdot & \cdot & \cdot & \cdot & \cdot & \cdot & \cdot & \cdot \\
\cdot & \cdot & \cdot & h_3 & h_4 & h_5 & \cdot & \cdot & \cdot & -h_2 & -h_1 & -1\\
\cdot & \cdot & \cdot & \cdot & 1 & h_1 & \cdot & \cdot & \cdot & \cdot & \cdot & \cdot \\
\cdot & \cdot & \cdot & \cdot & \cdot & \cdot & h_3-h_2 & h_4-h_1 & h_5-1 & \cdot & \cdot & \cdot\\
\cdot & \cdot & \cdot & \cdot & \cdot & \cdot & h_1-1 & h_2 & h_3 & \cdot & \cdot & \cdot\\
\cdot & \cdot & \cdot & \cdot & \cdot & \cdot & \cdot & 1 & h_1 & \cdot & \cdot & \cdot\\
\cdot & \cdot & \cdot & -h_2 & -h_1 & -1 & \cdot & \cdot & \cdot & h_3 & h_4 & h_5\\
\cdot & \cdot & \cdot & -1 & \cdot & \cdot  & \cdot & \cdot & \cdot & h_1 & h_2 & h_3\\
\cdot & \cdot & \cdot & \cdot & \cdot & \cdot & \cdot & \cdot & \cdot & 1 & h_1 & h_2\\
\cdot & \cdot & \cdot & \cdot & \cdot & \cdot & \cdot & \cdot & \cdot & \cdot & \cdot & 1
\end{vmatrix}.
\]
The top-left $3\times 3$ minor and central $3\times 3$ minor occupying rows 
6--8 and columns 7--9 are clearly equal to $\o_{(2,2)}$ and $\so_{(3,2,1)}^-$,
respectively.
One way to isolate the copy of $\so_{(3,2,1)}^-$ is to push it so that it is
the bottom-right $3\times 3$ submatrix, while preserving the order of the 
other rows and columns.
In this case such a procedure will introduce a sign of $-1$.
Putting this together, we have shown that
\[
\varphi_4\o_\la
=\o_{(2,2)}\so_{(3,2,1)}^-
\begin{vmatrix}
h_3 & h_4 & h_5 & -h_2 & -h_1 & -1 \\
\cdot & 1 & h_1 & \cdot & \cdot & \cdot \\
-h_2 & -h_1 & -1 & h_3 & h_4 & h_5 \\
-1 & \cdot & \cdot & h_1 & h_2 & h_3 \\
\cdot & \cdot & \cdot & 1 & h_1 & h_2 \\
\cdot & \cdot & \cdot & \cdot & \cdot & 1 \\
\end{vmatrix}.
\]
Our goal is to show that this final unidentified determinant is equal to 
$\rs_{(4,1),(2,1,1)}$.
Clearly the extra signs can be cleared by multiplying the
first two rows and first three columns by $-1$ each, 
generating an overall sign of $-1$.
Then one need only push the first column past the second and third, which
does not change the sign, and the 
resulting determinant is precisely a copy of $\rs_{(4,1),(2,1,1)}$.
Thus,
\[
\varphi_4\o_\la
=-\o_{(2,2)}\so_{(3,2,1)}^-\rs_{(4,1),(2,1,1)}.
\]
Note that $(-1)^{\varepsilon_{\la;12}^{\o}}=1$ so the overall sign clearly 
agrees with Theorem~\ref{Thm_AK_o}.
In the next sections we show that, with a little extra work, this argument 
also works in general for the universal characters
$\o_\la,\sp_\la$ and $\so_\la$.

\subsection{Proof of Theorem~\ref{Thm_AK_o}}
Let $\la$ be a partition such that $l(\la)\leq nt$ and consider the
definition \eqref{Eq_oh} of $\o_\la$
\[
\o_\la=\det_{1\leq i,j\leq nt}\big(h_{\la_i-i+j}-h_{\la_i-i-j}\big).
\]
We permute the rows and columns by $w_t(\la;nt)$ and $w_t(0;nt)$ respectively,
which introduces a sign of
\begin{equation}\label{Eq_first-sign}
(-1)^{\binom{n+1}{2}\binom{t}{2}}\sgn(w_t(\la;nt)).
\end{equation}
The modular behaviour of the indices of each row is now known.
There are three possibilities for the entries of $\varphi_t\o_\la$:
both $h$'s may survive, one $h$ may survive, or the entry is necessarily zero.
For both to survive, we see that 
$h_{\la_i-i+j}$ and $h_{\la_i-i-j}$ are nonzero under $\varphi_t$
if and only if $\la_i-i\equiv-j\equiv0\Mod{t}$ or, if $t$ is even,
$\la_i-i\equiv-j\equiv t/2\Mod{t}$.
In the first instance, by Lemma~\ref{Lem_phi},
\[
\varphi_t \big(h_{\la_i-i+j}-h_{\la_i-i-j}\big)
=h_{\lar{0}_k-k+\ell+m_0(\la;nt)-n}-h_{\lar{0}_k-k-\ell+m_0(\la;nt)-n},
\] 
where $1\leq k\leq m_0(\la;nt)$ and $1\leq \ell\leq n$.
Moreover, all other entries in the first $m_0(\la;nt)$ rows and $n$ columns
are zero.
If $t$ is even then we also find a submatrix of size $m_{t/2}(\la;nt)\times n$
in the rows $1+\sum_{r=0}^{(t-2)/2}m_r(\la;nt)$ to 
$\sum_{r=0}^{t/2}m_r(\la;nt)$ and columns $1+nt/2$ to $n(t+2)/2$.
The entries of this submatrix are 
\[
\varphi_t\big(h_{\la_i-i+j}-h_{\la_i-i-j}\big)
=h_{\lar{t/2}_k-k+\ell+m_{t/2}(\la;nt)-n}-
h_{\lar{t/2}_k-k-\ell+m_{t/2}(\la;nt)-n+1},
\]
where $1\leq k\leq m_{t/2}(\la;nt)$ and $1\leq\ell\leq n$.
Again, all other entries in these rows and columns are necessarily zero
under $\varphi_t$.
Given a row corresponding to the residue class $r$ where 
$1\leq r\leq\lfloor(t-1)/2\rfloor$, there are two possibilities for the entry to
potentially survive: the column corresponds to the residue class $r$ or $t-r$.
Again, by Lemma~\ref{Lem_phi},
\[
\varphi_t\big(h_{\la_i-i+j}-h_{\la_i-i-j}\big)
=\begin{cases}
h_{\lar{r}_k-k+\ell+m_r(\la;nt)-n} & \text{if $j\equiv -r\Mod{t}$}, \\
-h_{\lar{r}_k-k-\ell+m_r(\la;nt)-n+1} & \text{if $j\equiv r\Mod{t}$}. \\
\end{cases}
\]
The set of indices of the complete homogeneous symmetric functions in 
such a row are
\begin{subequations}\label{Eq_index-sets}
\begin{align}\label{Eq_index-set1}
&\big\{\lar{r}_k-k-\ell+m_r(\la,nt)+1\mid 1\leq \ell\leq 2n\big\}\\
&\,=\big\{\lar{r}_k-k+\ell\mid 1\leq \ell\leq m_r(\la,nt)\big\}
\sqcup\big\{\lar{r}_k-k-\ell+1\mid 1\leq\ell\leq m_{t-r}(\la,nt)\big\}.\notag
\intertext{If we look at the complementary row corresponding to $t-r$, 
then a similar computation shows that the indices are}
&\big\{\lar{t-r}_k-k-\ell+m_{t-r}+1\mid 1\leq\ell\leq 2n\big\} 
\label{Eq_index-set2}\\
&\,=\big\{\lar{t-r}_k-k+\ell\mid 1\leq \ell\leq m_{t-r}(\la,nt)\big\}
\sqcup\big\{\lar{t-r}_k-k-\ell+1\mid 1\leq\ell\leq m_r(\la,nt)\big\}.\notag
\end{align}
\end{subequations}
We have now identified the entries which do not necessarily vanish under
$\varphi_t$.
These can be rearranged into a block-diagonal matrix. 
If $t$ is even, we move the submatrix corresponding to $t/2$ to the 
bottom-right $m_{t-1}(\la;nt)$ rows and $n$ columns, which picks up a sign of
\[
(-1)^{m_{t/2}(\la;nt)\sum_{r=(t+2)/2}^{t-1}m_r(\la;nt)+n^2(t-2)/2}.
\]
We then group the rows and columns corresponding to the residue classes 
$r$ and $t-r$ together with $0\leq r \leq \lfloor(t-1)/2\rfloor$ increasing.
The determinant is now block-diagonal and the blocks have dimension
$m_0(\la;nt)\times n$, $(m_r(\la;nt)+m_{t-r}(\la;nt))\times 2n$ 
for $1\leq r\leq \lfloor (t-1)/2\rfloor$ and, if $t$ is even, $m_{t/2}(\la;nt)
\times n$.
Since the determinant of a block-diagonal matrix vanishes if one of the
blocks is not a square, we can therefore conclude that $\varphi_t\o_\la$ 
vanishes unless the conditions \eqref{Eq_m_conditions} with $z=-1$ hold in 
Lemma~\ref{Lem_AK_orth}, i.e., unless $\tcore(\la)$ is orthogonal.
In this case the top-left $n\times n$ minor is equal to 
$\o_{\lar{0}}$ and if $t$ is even the bottom-right minor corresponds to
$\so_{\lar{t/2}}^-$.
Note that in this case the grouping of the $2n\times 2n$ minors does not 
change the sign of the determinant since each row and column is pushed 
past an even number of rows or columns.
For each $1\leq r\leq \lfloor(t-1)/2\rfloor$ these final minors are of the form
\[
\begin{pmatrix}
h_{\lar{r}_1+m_r-n} & \cdots & h_{\lar{r}_1+m_r-1}
& -h_{\lar{r}_1+m_r-n-1} & \cdots & -h_{\lar{r}_1+m_r-2n} \\
\vdots & & \vdots & \vdots & & \vdots \\
h_{\lar{r}_{m_r}+1-n} & \cdots & h_{\lar{r}_{m_r}}
& -h_{\lar{r}_{m_r}-n} & \cdots & -h_{\lar{r}_{m_r}-2n+1} \\
-h_{\lar{t-r}_1+m_{t-r}-n-1} & \cdots & -h_{\lar{t-r}_1+m_{t-r}-2n}
& h_{\lar{t-r}_1+m_{t-r}-n} & \cdots & h_{\lar{t-r}_1+m_{t-r}-1} \\
\vdots & & \vdots & \vdots & & \vdots \\
-h_{\lar{t-r}_{m_{t-r}}-n} & \cdots & -h_{\lar{t-r}_{m_{t-r}}-2n+1}
& h_{\lar{t-r}_{m_{t-r}}+1-n} & \cdots & h_{\lar{t-r}_{m_{t-r}}}
\end{pmatrix},
\]
where we write $m_r=m_r(\la;nt)$.
Clearing the negatives in this minor produces the sign 
$(-1)^{m_r(\la;nt)+n}$.
If $m_r(\la;nt)=m_{t-r}(\la;nt)=n$ then we are done. 
If $m_r(\la;nt)>n$ then we need to move the
columns $n+1$ to $m_r(\la;nt)$ so they are the first 
$m_r(\la;nt)-n$ columns, and then
reverse the order. This gives a sign of 
\[
(-1)^{n(m_r(\la;nt)-n)+\binom{m_r(\la;nt)-n}{2}}
=(-1)^{\binom{m_r(\la;nt)}{2}-\binom{n}{2}}.
\]
If $m_r(\la;nt)<n$ then we need to push the $m_{t-r}(\la;nt)-n$ missing rows
past the $n-m_r$ rows to their right and then reverse again, giving the same 
sign
\[
(-1)^{(m_{t-r}(\la;nt)-n)m_r(\la;nt)+\binom{m_{t-r}(\la;nt)-n}{2}}
=(-1)^{\binom{m_r(\la;nt)}{2}-\binom{n}{2}},
\]
since $m_{t-r}(\la;nt)-n=n-m_r(\la;nt)$.
In each of the three cases the resulting determinant is equal to 
$\rs_{\lar{r},\lar{t-r}}$.
Collecting all of the above determinant manipulations, the value of 
$\varepsilon^{\o}_{\la;nt}$ is
\begin{multline*}
\frac{(n+1)nt(t-1)}{4}+
\sum_{r=1}^{\lfloor(t-1)/2\rfloor}
\bigg(\binom{m_r+1}{2}+\binom{n+1}{2}\bigg)
\\+\begin{cases}
n\sum_{r=1}^{(t-2)/2} m_{t-r}+\frac{n^2(t-2)}{2} & 
\text{$t$ even}, \\
0 & \text{$t$ odd}.
\end{cases}
\end{multline*}
To see that this agrees with the sign of Ayyer and Kumari, we use 
\eqref{Eq_o-rk} together with the fact 
that for odd $t$, $(t+1)(t-1)n(n+1)/4$ is even and for even $t$,
$n(n+1)(t^2-2)/4$ has the same parity as $n(n+1)/2$.
The above exponent therefore has the same parity as
\[
\varepsilon^{\o}_{\la;nt}
=\sum_{r=\lfloor(t+2)/2\rfloor}^{t-1}\binom{m_r(\la;nt)+1}{2}
+\rk(\tcore(\la))+\begin{cases}
\binom{n+1}{2}+n\rk(\tcore(\la)) & \text{$t$ even}, \\
0 & \text{$t$ odd}.
\end{cases}
\]
This completes the proof.

\subsection{Proof of Theorem~\ref{Thm_AK_sp}}\label{Sec_Proof3}
It is of course possible to prove Theorem~\ref{Thm_AK_sp} by direct 
manipulation of the $h$ Jacobi--Trudi-type formula for $\sp_\la$
\eqref{Eq_sph}.
However, it will be more insightful to
begin with the $e$ Jacobi--Trudi-type formula \eqref{Eq_sp-e}
\[
\sp_{\la}=\det_{1\leq i,j\leq nt}\big(e_{\la_i'-i+j}-e_{\la_i'-i-j}\big),
\]
where we assume that $n$ is an integer such that $nt\geq \la_1$.
We further assume that $nt\geq l(\la)$, since, in the end, our sign 
will be independent of $n$.
The values of $\varphi_t h_r$ and $\varphi_t e_r$ differ by a sign of 
$(-1)^{(t-1)r/t}$, and the indices of the $e$'s in this formula are the same 
as the $h$'s in the formula for $\o_{\la'}$ (\eqref{Eq_oh} with 
$\la\mapsto\la'$), so we can simply replace each $h$ 
by a signed $e$ in the previous proof.
Moreover, by Corollary~\ref{Cor_reverse}, we know that the $t$-quotient of 
$\la'$ is simply the reverse of the $t$-quotient of $\la$.
We can therefore already claim that $\varphi_t \sp_\la$ vanishes unless
$\tcore(\la)$ is symplectic, in which case
\[
\varphi_t\sp_\la
=(-1)^{\delta} \sgn(w_t(\la';nt))
\sp_{\lar{t-1}}\prod_{i=0}^{\lfloor(t-3)/2\rfloor}\rs_{\lar{i},\lar{t-i-2}}
\times\begin{cases}
\so_{\lar{(t-2)/2}} & \text{$t$ even}, \\
1 & \text{$t$ odd},
\end{cases}
\]
where 
\begin{multline*}
\delta=(t-1)\sum_{r=0}^{t-1}\abs{\lar{r}}+
\sum_{r=\lfloor(t+1)/2\rfloor}^{t-1}\binom{m_r(\la';nt)}{2}
\\+\rk(\tcore(\la))+\begin{cases}
\binom{n+1}{2}+n\rk(\tcore(\la)) & \text{$t$ even}, \\
0& \text{$t$ odd}.
\end{cases}
\end{multline*}
All that remains now is to show that this sign agrees with that of 
Theorem~\ref{Thm_AK_sp}.
By a combination of Lemmas~\ref{Lem_signs1} and \ref{Lem_signs2} we may replace
$\sgn(w_t(\la';nt))$ by $\sgn(w_t(\la;nt))$, which cancels 
$\rk(\tcore(\la))+(t-1)\sum_{r=0}^{t-1}\abs{\lar{r}}$ in $\delta$.
If we call this new exponent $\delta'$, then we
also have by Lemma~\ref{Lem_reverse} that $m_r(\la';nt)=m_{r-1}(\la;nt)$
for $\lfloor(t+1)/2\rfloor\leq r\leq t-1$, which implies
$\delta'=\varepsilon_{\la;nt}^{\sp}$.

\subsection{Proof of Theorem~\ref{Thm_AK_so}}
The final proof closely follows the first. 
Let $\la$ be a partition of length at most $nt$ and consider
\[
\mathrm{so}_\la=\det_{1\leq i,j\leq nt}\big(h_{\la_i-i+j}+h_{\la_i-i-j+1}\big).
\]
As before we apply the permutations $w_t(\la;nt)$ and $w_t(0;nt)$ to the 
rows and columns of this determinant, introducing the sign 
\eqref{Eq_first-sign}.
Unlike before, there is only one case in which both $h$'s may survive. 
If $t$ is odd and $\la_i-i\equiv-j\equiv (t-1)/2\Mod{t}$ then we have
\begin{multline*}
\varphi_t\big(h_{\la_i-i+j}+h_{\la_i-i-j+1}\big)
\\[1mm]=h_{\lar{(t-1)/2}_k-k+\ell+m_{(t-1)/2}(\la;nt)-n}
+h_{\lar{(t-1)/2}_k-k-\ell+1+m_{(t-1)/2}(\la;nt)-n}.
\end{multline*}
where $1\leq k\leq m_{(t-1)/2}(\la;nt)$ and $1\leq\ell\leq n$.
These entries lie in the rows $1+\sum_{r=0}^{(t-3)/2}m_r(\la;nt)$ to
$\sum_{r=0}^{(t-1)/2}m_r(\la;nt)$ and columns
$1+n(t-1)/2$ to $n(t+1)/2$, and outside of their intersection, all other entries
in these rows and columns are zero.
Now consider a row corresponding to the residue class $r$ for 
$0\leq r\leq \lfloor(t-2)/2\rfloor$.
Then the column must fall into the residue class $r$ or $t-r-1$ in order for
the entry to not necessarily vanish.
In this case we now have
\[
\varphi_t\big(h_{\la_i-i+j}+h_{\la_i-i-j+1}\big)
=\begin{cases}
h_{\lar{r}_k-k+\ell_1+m_r(\la;nt)-n} & \text{if $j\equiv -r\Mod{t}$}, \\
h_{\lar{r}_k-k-\ell_1+m_r(\la;nt)-n+1} & \text{if $j\equiv r+1\Mod{t}$}. \\
\end{cases}
\]
Again, a similar computation holds for the row corresponding to $t-r-1$,
and the sets of indices agree with \eqref{Eq_index-sets} but with 
$t-r\mapsto t-r-1$ in \eqref{Eq_index-set2}.
If $t$ is odd we move the central submatrix corresponding to $(t-1)/2$ to
the top-left, picking up a sign of
\[
(-1)^{m_{(t-1)/2}(\la;nt)\sum_{r=0}^{(t-3)/2}m_r(\la;nt)+n^2(t-1)/2}.
\]
The grouping and rearrangement of the remaining minors is the same as in the
first proof above. We only remark that the result is the determinant of a 
block-diagonal matrix with blocks of dimensions 
$(m_r(\la;nt)+m_{t-r-1}(\la;nt))\times 2n$ for 
$0\leq r\leq \lfloor(t-2)/2\rfloor$ plus one of size
$m_{(t-1)/2}(\la;nt)\times n$ if $t$ is odd.
Thus the determinant vanishes unless 
\eqref{Eq_m_conditions} holds with $z=0$, i.e., unless $\tcore(\la)$ is 
self-conjugate.
Accounting for the sign of $(-1)^{\binom{m_r(\la;nt)}{2}+\binom{n}{2}}$ 
from reordering the columns in the copies of $\rs_{\lar{r},\lar{t-r-1}}$,
the exponent $\varepsilon^{\so}_{\la;nt}$ has the value
\[
\frac{(n+1)nt(t-1)}{4}
+\sum_{r=0}^{\lfloor(t-2)/2\rfloor}\bigg(\binom{m_r}{2}+\binom{n}{2}\bigg)
+\begin{cases}
0 & \text{$t$ even},\\
n\sum_{r=0}^{(t-3)/2}m_r+n^2(t-1)/2 & \text{$t$ odd}.
\end{cases}
\]
By \eqref{Eq_so-rk} this has the same parity as
\[
\varepsilon^{\so}_{\la;nt}=
\sum_{r=\lfloor (t+1)/2\rfloor}^{t-1}\binom{m_r(\la;nt)+1}{2}+
\begin{cases}
0 & \text{$t$ even}, \\
n\rk(\tcore(\la)) & \text{$t$ odd}.
\end{cases}
\]

\subsection{Proof of Proposition~\ref{Prop_signs}}
To close out this section, we sketch the proof of Proposition~\ref{Prop_signs}.
In the Schur case, by Corollary~\ref{Cor_reverse} and the fact that 
$\la/\mu$ is $t$-tileable if and only if $\la'/\mu'$ is, we already have
$\omega\varphi_t s_{\la/\mu}=\pm\varphi_t\omega s_{\la/\mu}$.
The precise difference in sign is then provided by Lemma~\ref{Lem_signs1}.
Again using Corollary~\ref{Cor_reverse}, we have
\begin{align*}
\omega\varphi_t\so_\la
&=(-1)^{\varepsilon^{\so}_{\la;nt}} \sgn(w_t(\la;nt))
\prod_{r=0}^{\lfloor(t-2)/2\rfloor}
\!\rs_{(\lar{r})',(\lar{t-r-1})'}
\!\times\!\begin{cases}
1 & \text{$t$ even}, \\
\so_{(\lar{(t-1)/2})'} & \text{$t$ odd},
\end{cases}\\
&=(-1)^{\varepsilon^{\so}_{\la;nt}+\varepsilon^{\so}_{\la';nt}}
\sgn(w_t(\la;nt))\sgn(w_t(\la';nt))
\varphi_t\so_{\la'},
\end{align*}
where $n$ should be large enough so that $\la$ is contained in an $nt\times nt$
box.
Combining Lemmas~\ref{Lem_signs0} and \ref{Lem_signs1} shows that, in this case,
\[
\sgn(w_t(\la;nt))\sgn(w_t(\la';nt))=
(-1)^{(t-1)(\abs{\lar{0}}+\cdots+\abs{\lar{t-1}})}.
\]
Moreover, $\varepsilon_{\la;nt}^{\so}=\varepsilon_{\la';nt}^{\so}$, so
that the total sign agrees with the claim.

\section{Other factorisations}\label{Sec_stuff}
\subsection{Littlewood-type factorisations}
In \cite[\S7.3]{Littlewood40}, Littlewood proves a factorisation slightly
more general than the one contained in Theorem~\ref{Thm_skew} for $\mu$ 
empty; see also \cite[Theorem~2.7]{AK22}.
Here, and below, we let $\lar{r}=\lar{k}$ if $k\equiv r\Mod{t}$.
\begin{theorem}\label{Thm_extra-y}
Let $\la$ be a partition of length at most $nt+1$ and 
$X=(x_1,\dots,x_n)$ a set of variables. Then for another variable $q$,
\[
s_\la(X,\zeta X,\dots,\zeta^{t-1}X,q)=0
\]
unless $\tcore(\la)=(c)$ for some $0\leq c\leq t-1$, in which case
\[
s_\la(X,\zeta X,\dots,\zeta^{t-1}X,q)=
\sgn_t(\la/(c))q^cs_{\lar{c-1}}(X^t,q^t)\prod_{\substack{r=0\\r\neq c-1}}^{t-1}
s_{\lar{r}}(X^t).
\]
\end{theorem}
This theorem can also be placed in our framework, however in a somewhat less
elegant manner than our other results.
The operator 
$\varphi_t^q:\La\longrightarrow \La\otimes_{\mathbb{Z}}\mathbb{Z}[q]$ 
which gives the above may be defined by
\[
\varphi_t^q h_{at+b}:=q^{b}\sum_{k\geq 0}q^{kt} h_{a-k}
=\sum_{k\geq0}q^k\varphi_t h_{at+b-k}.
\]
Note that the sums are finite since $h_r$ vanishes for negative $r$, and that
for $q=0$ this reduces to the operator $\varphi_t$ from the earlier sections.
Alternatively, since 
\[
h_r(X,q)=\sum_{k\geq0}q^k h_{r-k}(X),
\]
the image of $\varphi_t^q$ acting on any symmetric function $f$ is the same as
the image of $\varphi_t$ acting on $f(X,q)$, where $\varphi_t$ acts only 
on the $X$ variables.
Using $\varphi_t^q$, 
Littlewood's above theorem may be phrased as follows.
After the statement we provide a short proof which relies only
on Theorem~\ref{Thm_skew}.
\begin{prop}
We have that $\varphi_t^q s_\la=0$ unless $\tcore(\la)=(c)$ for some 
$c$ such that $0\leq c\leq t-1$, in which case
\[
\varphi_t^q s_\la=
\sgn_t(\la/(c))q^c\prod_{\substack{r=0\\r\neq c-1}}^{t-1}s_{\lar{r}}
\sum_{k\geq0}q^{kt}s_{\lar{c-1}/(k)}.
\]
\end{prop}
\begin{proof}
The first observation is that 
\[
\varphi_t^qs_\la=\sum_{k\geq 0} q^k\varphi_ts_{\la/(k)},
\]
which is a simple consequence of the branching rule for Schur functions
\cite[p.~72]{Macdonald95}.
In the case that $l(\tcore(\la))>1$ 
then each term in the sum on the right-hand side vanishes
by Theorem~\ref{Thm_skew} as the $t$-cores of the inner and outer shape
can never be equal.
Now assume $\tcore(\la)=(c)$ for some $0\leq c\leq t-1$, which is a complete
set of $t$-cores with length one.
Then the nonzero terms in the sum on the right-hand side are those for
which $k$ is of the form $\ell t+c$ with $\ell\geq0$ and $\ell t+c\leq 
\la_1$.
Therefore 
\[
\varphi_t^qs_\la=q^{c}\sum_{\ell\geq0}q^{\ell t}
\varphi_ts_{\la/(\ell t+c)}
=q^c\sum_{\ell\geq0}\sgn_t(\la/(\ell t+c))q^{\ell t}s_{\lar{c-1}/(\ell)}
\prod_{\substack{r=0\\r\neq c-1}}^{t-1}s_{\lar{r}},
\]
again by Theorem~\ref{Thm_skew}. By our convention we always
compute the $t$-quotient using a beta set with number of elements a multiple 
of $t$. This means that the single row $(\ell t+c)$ has one non-empty
element in its $t$-quotient, $\lar{c-1}$.
Moreover, since the partitions $(\ell t+c)$ all differ by a ribbon of
height zero, the sign of each term in the sum is the same and equal to
$\sgn_t(\la/(c))$.
Putting all of this together, we arrive at
\[
\varphi_t^qs_\la=
\sgn_t(\la/(c))q^c\prod_{\substack{r=0\\r\neq c-1}}^{t-1}s_{\lar{r}}
\sum_{\ell\geq0}q^{\ell t}s_{\lar{c-1}/(\ell)}.\qedhere
\]
\end{proof}

We do not see, at this stage, whether it is possible to extend the previous
result to skew Schur functions.
If we expand $\varphi_t^q s_{\la/\mu}$ as in the proof above, we find that
\[
\varphi_t^q s_{\la/\mu}
=\sum_{\nu\succ\mu}q^{\abs{\nu}-\abs{\mu}}\varphi_t s_{\la/\nu},
\]
where $\nu\succ\mu$ means that $\nu/\mu$ is a \emph{horizontal strip},
i.e., $\nu\supseteq\mu$ and $\nu/\mu$ contains at most one box in each column 
of its Young diagram.
Of course, this implies that $\varphi_t^qs_{\la/\mu}=0$ if there does not
exist a $\nu$ such that $\nu\succ\mu$ and $\la/\nu$ is $t$-tileable.
However, the sum may vanish even if such a $\nu$ exists.
For example,
\begin{align*}
\varphi_2^qs_{(4,4)/(1)}&=q\varphi_t\big(s_{(4,4)/(2)}+s_{(4,4)/(1,1)}\big)
+q^3\varphi_t\big(s_{(4,4)/(4)}+s_{(4,4)/(3,1)}\big) \\
&=q(s_{(2)}s_{(2)/(1)}-s_{(2)}s_{(2)/(1)})
+q^3(s_{(2)}-s_{(2)}) \\
&=0.
\end{align*}

In a similar direction Pfannerer \cite[Theorem~4.4]{Pfannerer22} has shown 
that, if $\la$ has empty $t$-core and $m=\ell t+k$ is any integer, then the 
Schur function $s_\la(1,\zeta,\dots,\zeta^{m-1})$ always factors as a product 
of Schur functions with variables all one indexed by the $t$-quotient of $\la$.
When $m$ is a multiple of $t$ this becomes a special case of Littlewood's 
theorem (Theorem \ref{Thm_skew} with $\mu$ empty) noted by Reiner, Stanton and 
White \cite[Theorem~4.3]{RSW04}.
Pfannerer's result has subsequently been generalised by Kumari 
\cite[Theorem~2.2]{Kumari22b}, in addition to analogues of 
Theorem~\ref{Thm_extra-y} for other classical group characters.
It is an open problem to see how these factorisations fit into our story.

\subsection{Factorisations of supersymmetric Schur functions}
Recently, Kumari has given a version of Theorem~\ref{Thm_skew} for the
so-called \emph{skew hook Schur functions} (or
\emph{supersymmetric skew Schur functions}) \cite[Theorem~3.2]{Kumari22}.
For two independent sets of variables (\emph{alphabets}), we denote their
plethystic difference by $X-Y$; see, e.g., 
\cite{Haglund08,Lascoux03} for the necessary background on plethystic 
notation. We also note that for an alphabet $X$, we let $\varepsilon X$ be
the alphabet with all variables negated.
The \emph{complete homogeneous supersymmetric function} used in
\cite{Kumari22} may be defined as\[
\sum_{j=0}^r h_j(X)e_{r-j}(Y)=h_r[X-\varepsilon Y].
\]
The hook Schur function is then the Jacobi--Trudi determinant of these 
functions, so that
\[
s_{\la/\mu}[X-\varepsilon Y]=\det_{1\leq i,j\leq n}\big(
h_{\la_i-\mu_j-i+j}[X-\varepsilon Y]\big).
\]
From this, it follows readily that Kumari's factorisation for the
hook Schur functions is contained in Theorem~\ref{Thm_skew} above at the
alphabet $X-\varepsilon Y$.

\subsection{Factorisations of $\rs_{\la,\mu}$}\label{Sec_CK}
To close, we point out that the universal character $\rs_{\la,\mu}$ can be
used to lift some factorisation results, discovered by Ciucu and Krattenthaler
\cite[Theorems~3.1--3.2]{CK09} and subsequently generalised by
Ayyer and Behrend \cite[Theorems~1--2]{AB19}, to the universal character level.
In the next result we write $\la+1^n=(\la_1+1,\dots,\la_n+1)$ 
where $n\geq l(\la)$.
\begin{theorem}\label{Thm_CK}
For $\la$ a partition of length at most $n$, there holds
\begin{subequations}\label{Eq_CK}
\begin{align}\label{Eq_CK1}
\rs_{\la,\la}&=\so_\la\so_\la^-,
\intertext{and}
\rs_{\la+1^n,\la}&=\o_{\la+1^n}\sp_\la.\label{Eq_CK2}
\end{align}
\end{subequations}
Moreover, for $\la$ a partition of length at most $n+1$,
\begin{subequations}\label{Eq_CK'}
\begin{align}\label{Eq_CK3}
\rs_{(\la_1,\dots,\la_n),(\la_2,\dots,\la_{n+1})}
+&\rs_{(\la_1-1,\dots,\la_{n+1}-1),(\la_2+1,\dots,\la_n+1)}\\\notag
&\qquad\qquad\quad
=\sp_{(\la_1,\dots,\la_n)}\o_{(\la_2,\dots,\la_{n+1})},
\intertext{and}
\label{Eq_CK4}
\rs_{(\la_1+1,\dots,\la_n+1),(\la_2,\dots,\la_{n+1})}
+&\rs_{(\la_1,\dots,\la_{n+1}),(\la_2+1,\dots,\la_n+1)}\\\notag
&\qquad\qquad\quad
=\so_{(\la_1+1,\dots,\la_n+1)}\so_{(\la_2,\dots,\la_{n+1})}^-.
\end{align}
\end{subequations}
\end{theorem}
To get back to the results of Ayyer and Behrend one simply evaluates
both sides of each equation at the alphabet $(x_1^\pm,\dots,x_n^\pm)$.
The precise forms present in \cite[Equations~(18)--(21)]{AB19} then follow 
from \eqref{Eq_det-rep}.\footnote{The factor of $(1+\delta_{0,\la_{n+1}})$
in \cite[Equation~(20)]{AB19} is not present in our generalisation 
\eqref{Eq_CK3} since the second character vanishes if $\la_{n+1}=0$.}
As identities for Laurent polynomials the pairs of identities 
\eqref{Eq_CK} and \eqref{Eq_CK'} admit uniform statements.
However no such uniform statement will exist for the above generalisation,
since this requires characters indexed by half-partitions, which cannot be
handled by the universal characters.
Ayyer and Fischer \cite{AF20} have also given skew analogues of the 
non-universal case of Theorem~\ref{Thm_CK}.
Jacobi--Trudi formulae for the symplectic and orthogonal characters have
recently been derived in \cite{AFHS23,JLW22}, and so there are candidates
for the universal characters for those objects.
However, the main obstacle in lifting Ayyer and Fischer's results to the 
universal level is the lack of a skew analogue of $\rs_{\la,\mu}$.

\begin{proof}[Proof of Theorem~\ref{Thm_CK}]
First up is \eqref{Eq_CK1}, which is the simplest of the four. 
In the determinant
\[
\rs_{\la,\la}=
\det_{1\leq i,j\leq 2n}
\begin{pmatrix}
\big(h_{\la_i-i+j}\big)_{1\leq i,j\leq n} & 
\big(h_{\la_i-i-j+1}\big)_{\substack{1\leq i,j\leq n}} \\[2mm]
\big(h_{\la_i-i-j+1}\big)_{\substack{1\leq i,j\leq n}} &
\big(h_{\la_i-i+j}\big)_{1\leq i,j\leq n}
\end{pmatrix},
\]
add the blocks on the right to the blocks on the left, and then subtract the
blocks on the top from the blocks on the bottom, giving
\begin{align*}
\rs_{\la,\la}&=
\det_{1\leq i,j\leq 2n}
\begin{pmatrix}
\big(h_{\la_i-i+j}+h_{\la_i-i-j+1}\big)_{1\leq i,j\leq n} & 
\big(h_{\la_i-i-j+1}\big)_{\substack{1\leq i,j\leq n}} \\[2mm]
0 &
\big(h_{\la_i-i+j}-h_{\la_i-i-j+1}\big)_{1\leq i,j\leq n}
\end{pmatrix} \\
&=\so_\la\so_\la^-.
\end{align*}
For the second identity \eqref{Eq_CK2},
\[
\rs_{\la+1^n,\la}=
\det_{1\leq i,j\leq 2n}
\begin{pmatrix}
\big(h_{\la_i-i+j+1}\big)_{1\leq i,j\leq n} & 
\big(h_{\la_i-i-j+2}\big)_{\substack{1\leq i,j\leq n}} \\[2mm]
\big(h_{\la-i-j+1}\big)_{\substack{1\leq i,j\leq n}} &
\big(h_{\la_i-i+j}\big)_{1\leq i,j\leq n}
\end{pmatrix},
\]
and we add columns $1, \dots, n-1$ to the columns $n+2,\dots,2n$
and then subtract the bottom two blocks from the top two, resulting in
\begin{align*}
&\rs_{\la+1^n,\la}\\
&\quad=
\frac{1}{2}\det_{1\leq i,j\leq 2n}
\begin{pmatrix}
\big(h_{\la_i-i+j+1}-h_{\la_i-i-j+1}\big)_{1\leq i,j\leq n} & 
0 \\
\big(h_{\la-i-j+1}\big)_{\substack{1\leq i,j\leq n}} &
\big(h_{\la_i-i+j}+h_{\la_i-i-j+2}\big)_{1\leq i,j\leq n}
\end{pmatrix} \\[2mm]
&\quad=\o_{\la+1^n}\sp_\la.
\end{align*}
In the third identity, we consider the second determinant in the sum
in \eqref{Eq_CK3}
\begin{align*}
\det_{1\leq i,j\leq 2n}
\begin{pmatrix}
\big(h_{\la_i-i+j-1}\big)_{1\leq i,j\leq n+1} & 
\big(h_{\la_i-i-j}\big)_{\substack{1\leq i\leq n+1\\ 1\leq j\leq n-1}} \\[3mm]
\big(h_{\la_{i+1}-i-j+2}\big)_{\substack{1\leq i\leq n-1\\1\leq j\leq n+1}} &
\big(h_{\la_{i+1}-i+j+1}\big)_{1\leq i,j\leq n-1}
\end{pmatrix}.
\end{align*}
Push the first column so it becomes the $(n+1)$-st, and then push
the $(n+1)$-st row to the final row, which picks up a minus sign. 
The resulting determinant differs from that of 
$\rs_{(\la_1,\dots,\la_n),(\la_2,\dots,\la_{n+1})}$ 
in only the last row, so we can take the sum of the two, giving
\begin{align*}
\det_{1\leq i,j\leq 2n}
\begin{pmatrix}
\big(h_{\la_i-i+j}\big)_{1\leq i,j\leq n} & 
\big(h_{\la_i-i-j+1}\big)_{\substack{1\leq i\leq n}} \\[2mm]
\big(h_{\la_{i+1}-i-j+1}\big)_{\substack{1\leq i\leq n-1\\1\leq j\leq n}} &
\big(h_{\la_{i+1}-i+j}\big)_{\substack{1\leq i\leq n-1\\1\leq j\leq n}}\\[3mm]
(h_{\la_{n+1}-n-j+1}-h_{\la_{n+1}-n+j-1})_{1\leq j\leq n} & 
(h_{\la_{n+1}-n+j}-h_{\la_{n+1}-n-j})_{1\leq j\leq n}
\end{pmatrix}.
\end{align*}
In this new determinant, add columns $n+1,\dots,2n-1$ to columns $2,\dots,n$,
and then subtract rows $2,\dots,n$ from rows $n+1,\dots,2n-1$, which gives
\begin{align*}
\frac{1}{2}
\det_{1\leq i,j\leq 2n}
\begin{pmatrix}
\big(h_{\la_i-i+j}+h_{\la_i-i-j+2}\big)_{1\leq i,j\leq n} & 
\big(h_{\la_i-i-j+1}\big)_{\substack{1\leq i\leq n}} \\[2mm]
0 &
\big(h_{\la_{i+1}-i+j}-h_{\la_{i+1}-i-j}\big)_{1\leq i,j\leq n}\\[3mm]
\end{pmatrix},
\end{align*}
which equals
$\sp_{(\la_1,\dots,\la_n)}\o_{(\la_2,\dots,\la_{n+1})}$.
The final factorisation \eqref{Eq_CK4} follows almost identically.
\end{proof}

\subsection*{Acknowledgements}
I would like to thank Christian Krattenthaler for many useful 
comments, corrections and suggestions.
This work also benefitted from discussions with Ilse Fischer, 
Hans H\"ongesberg, Josef K\"ustner, Florian Schreier-Aigner and Ole Warnaar.


\begin{thebibliography}{99}
\bibitem{AFHS23}
S. P. Albion, I. Fischer, H. H{\"o}ngesberg and F. Schreier-Aigner,
\textit{Skew symplectic and orthogonal characters through lattice paths},
\href{https://arxiv.org/abs/2305.11730}{arXiv:2305.11730}.

\bibitem{APRU21}
P. Alexandersson, S. Pfannerer, M. Rubey and J. Uhlin,
\textit{Skew characters and cyclic sieving},
Forum Math. Sigma \textbf{9} (2021), Paper No. e41, 31pp.

\bibitem{AB19}
A. Ayyer and R. E. Behrend,
\textit{Factorization theorems for classical group characters, with 
applications to alternating sign matrices and plane partitions},
J. Combin. Theory Ser. A \textbf{165} (2019), 78--105.

\bibitem{AF20}
A. Ayyer and I. Fischer,
\textit{Bijective proofs of skew Schur polynomial factorizations},
J. Combin. Theory Ser. A \textbf{174} (2020), 40pp.

\bibitem{AK22}
A. Ayyer and N. Kumari,
\textit{Factorization of classical characters twisted by roots of unity},
J. Algebra \textbf{609} (2022), 437--483.

\bibitem{BR01}
J. Baik and E. M. Rains,
\textit{Algebraic aspects of increasing subsequences},
Duke Math. J. \textbf{109} (2001), 1--65.

\bibitem{CGR84}
Y. M. Chen, A. M. Garsia and J. Remmel,
\textit{Algorithms for plethysm},
Contemp. Math. \textbf{34} (1984), 109--153.

\bibitem{CK09}
M. Ciucu and C. Krattenthaler,
\textit{A factorization theorem for classical group characters, with 
applications to plane partitions and rhombus tilings},
Advances in Combinatorial Mathematics, 39--59, Springer, Berlin, 2009.

\bibitem{EPW14}
A. Evseev, R. Paget and M. Wildon,
\textit{Character deflations and a generalization of the Murnaghan--Nakayama
rule},
J. Group Theory \textbf{17} (2014), 1035--1070.

\bibitem{Farahat53}
H. Farahat,
\textit{On $p$-quotients and star diagrams of the symmetric group},
Proc. Cambridge Phil. Soc. \textbf{49} (1953), 157--160.

\bibitem{Farahat54}
H. Farahat,
\textit{On the representations of the symmetric group},
Proc. London Math. Soc. \textbf{4} (1954), 303--316.

\bibitem{Farahat58}
H. Farahat,
\textit{On Schur functions},
Proc. London Math. Soc. \textbf{8} (1958), 621--630.

\bibitem{GKS90}
F. Garvan, D. Kim and D. Stanton,
\textit{Cranks and $t$-cores},
Invent. Math. \textbf{101} (1990), 1--17.

\bibitem{Haglund08}
J. Haglund,
\textit{The $q,t$-Catalan Numbers and the Space of Diagonal Harmonics},
Univ. Lecture Ser., vol. 41, American Mathematical Society, Providence,
RI, 2008.

\bibitem{HJ11}
G.-N. Han and K. Q. Ji,
\textit{Combining hook length formulas and BG-ranks for partitions via the 
Littlewood decomposition},
Trans. Amer. Math. Soc. \textbf{363} (2011), 1041--1060.

\bibitem{JK81}
G. James and A. Kerber,
\textit{The representation theory of the symmetric group},
Encyclopedia of mathematics and its applications, Vol. 16, 
Addison-Wesley, Reading MA, 1981.

\bibitem{JLW22}
N. Jing, Z. Li and D. Wang,
\textit{Skew-type symplectic/orthogonal Schur functions},
\href{https://arxiv.org/abs/2208.05526}{arXiv:2208.05526}.

\bibitem{Koike89}
K. Koike,
\textit{On the decomposition of tensor products of the representations of the
classical groups: By means of the universal characters},
Adv. Math. \textbf{74} (1989), 57--86.

\bibitem{Koike97}
K. Koike,
\textit{Representations of spinor groups and the difference characters of 
$SO(2n)$},
Adv. Math. \textbf{128} (1997), 40--81.

\bibitem{KT87}
K. Koike and I. Terada,
\textit{Young-diagrammatic methods for the representation theory of the 
classical groups of type $B_n$, $C_n$, $D_n$},
J. Algebra \textbf{107} (1987), 466--511.

\bibitem{Kumari22}
N. Kumari,
\textit{Skew hook Schur functions and the cyclic sieving phenomenon},
\href{https://arxiv.org/abs/2211.14093}{arXiv:2211.14093}.

\bibitem{Kumari22b}
N. Kumari,
\textit{Factorization of classical characters twisted by roots of unity: II},
\href{https://arxiv.org/abs/2212.12477}{arXiv:2212.12477}.

\bibitem{Lascoux03}
A. Lascoux,
\textit{Symmetric Functions and Combinatorial Operators on Polynomials},
CMBS Reg. Conf. Ser. Math., vol. 99, American Mathematical Society,
Providence, RI, 2003.

\bibitem{Lascoux?}
A. Lascoux,
\textit{Symmetric functions}, unpublished notes
\href{https://www.emis.de/journals/SLC/wpapers/s68vortrag/ALCoursSf2.pdf}
{(link)}.

\bibitem{LRW20}
C.-h. Lee, E. M. Rains and S. O. Warnaar,
\textit{An elliptic hypergeometric function approach to branching rules},
SIGMA \textbf{16} (2020), paper 142, 52pp.

\bibitem{vanLeeuwen99}
M. A. A. van Leeuwen,
\textit{Edge sequences, ribbon tableaux, and an action of affine permutations},
European J. Combin. \textbf{20} (1999), 179--195.

\bibitem{Littlewood40}
D. E. Littlewood,
\textit{The Theory of Group Characters and Matrix Representations of
Groups}, Oxford University Press, London, 1940.

\bibitem{Littlewood44}
D. E. Littlewood,
\textit{On invariant theory under restricted groups},
Philos. Trans. Roy. Soc. London Ser. A \textbf{239} (1944), 387--417.

\bibitem{Littlewood51}
D. E. Littlewood,
\textit{Modular representations of symmetric groups},
Proc. Roy. Soc. London Ser. A \textbf{209} (1951), 333--353.

\bibitem{LR11}
N. A. Loehr and J. B. Remmel,
\textit{A computational and combinatorial expos\'e of plethystic calculus},
J. Algebraic Combin. \textbf{33} (2011), 163--198.

\bibitem{Macdonald95}
I. G. Macdonald,
\textit{Symmetric Functions and Hall Polynomials},
The Clarendon Press, Oxford University Press, New York, 1995.

\bibitem{Mizukawa02}
H. Mizukawa,
\textit{Factorization of Schur's $Q$-functions and plethysm},
Ann. Comb. \textbf{6} (2002), 87--101.

\bibitem{Nakayama40}
T. Nakayama,
\textit{On some modular properties of irreducible representations of a 
symmetric group II},
Jpn. J. Math. \textbf{17} (1940), 411--423.

\bibitem{NO51}
T. Nakayama and M. Osima,
\textit{Note on blocks of symmetric groups},
Nagoya Math. J. \textbf{2} (1951), 111--117.

\bibitem{Pak00}
I. Pak,
\textit{Ribbon tile invariants},
Trans. Amer. Math. Soc. \textbf{352} (2000), 5525--5561.

\bibitem{Pfannerer22}
S. Pfannerer,
\textit{A refinement of the Murnaghan--Nakayama rule by descents for 
border strip tableaux},
Comb. Theory \textbf{2} (2022), Paper No. 16, 12pp.

\bibitem{Prasad16}
D. Prasad,
\textit{A character relationship on $\mathrm{GL}_n(\mathbb{C})$},
Israel J. Math. \textbf{211} (2016), 257--270.

\bibitem{RSW04}
V. Reiner, D. Stanton and D. White,
\textit{The cyclic sieving phenomenon},
J. Combin. Theory Ser. A \textbf{108} (2004), 17--50.

\bibitem{Robinson48}
G. de B. Robinson,
\textit{On the representations of the symmetric group III},
Amer. J. Math. \textbf{70} (1948), 277--294.

\bibitem{Staal50}
R. A. Staal,
\textit{Star diagrams and the symmetric group},
Can. J. Math. \textbf{2} (1950), 79--92.

\bibitem{Stembridge87}
J. R. Stembridge,
\textit{Rational tableaux and the tensor algebra of $gl_n$},
J. Combin. Theory Ser. A \textbf{46} (1987), 79--120.

\bibitem{WW20}
A. Walsh and S. O. Warnaar,
\textit{Modular Nekrasov--Okounkov formulas},
S\'em. Lothar. Combin. \textbf{81} (2020), B81c, 28pp.

\bibitem{Weyl39}
H. Weyl,
\textit{The Classical Groups: Their Invariants and Representations},
Princeton University Press, Princeton, N.J., 1939.

\bibitem{Wildon22}
M. Wildon,
\textit{Review of the article ``A character relationship between symmetric
group and hyperoctahedral group'' by F. L\"ubeck and D. Prasad},
Mathematical Reviews \textbf{MR4196561} (2021).

\end{thebibliography}
\end{document}